\documentclass[reqno]{amsart}

\usepackage{mathrsfs, amsfonts, amsmath, amssymb, amscd}
\usepackage{amsbsy,bm, amsthm, enumerate}

\usepackage[english]{babel}
\usepackage{amsmath,amssymb,amsthm}
\usepackage{hyperref}
\usepackage[centering]{geometry}
\usepackage{xcolor}

\renewcommand{\epsilon}{\varepsilon}            

\newtheorem{theorem}{Theorem}[section]   
\newtheorem*{theorem*}{Theorem}          
\newtheorem{lemma}[theorem]{Lemma}
\newtheorem{proposition}[theorem]{Proposition}

\theoremstyle{definition}
\newtheorem{definition}[theorem]{Definition}

\newtheorem{example}[theorem]{Example}

\newtheorem{remark}{Remark}[section]

\numberwithin{equation}{section}

\title[On Hopf hypersurfaces of the complex hyperbolic quadric]
{On Hopf hypersurfaces of the complex hyperbolic quadric with constant principal curvatures}

\author{Haizhong Li, Hiroshi Tamaru \and Zeke Yao}

\makeatletter
\@namedef{subjclassname@2020}{\textup{2020} Mathematics Subject Classification} 
\makeatother


\date{}
\begin{document}
\begin{abstract}
In this paper, we study the Hopf hypersurfaces of the complex hyperbolic quadric
$Q^{m*}=SO^o_{2,m}/(SO_2\times SO_m)$ ($m\geq3$) with constant principal curvatures.
We classify the Hopf hypersurfaces of $Q^{m*}$ ($m\geq3$) with at most two
distinct constant principal curvatures. For Hopf hypersurfaces with three or four distinct constant principal curvatures, we determine the values of the principal curvatures as well as their multiplicities.
\end{abstract}

\maketitle

\section{Introduction}\label{sect:1}

As the non-compact dual of the complex quadric $Q^m$, the $m$-dimensional complex hyperbolic quadric
$Q^{m\ast}$ is a simply connected Riemannian symmetric space whose curvature tensor is
the negative of that of $Q^m$. It is known that $Q^{m\ast}$ can not be realized as a
homogeneous complex hypersurface of the complex hyperbolic space $\mathbb{C}H^{m+1}$
(see \cite{Sm-2}), whereas Montiel-Romero \cite{M-R} proved that $Q^{m*}$ can be isometrically
immersed into the complex anti-de Sitter space $\mathbb{C}H_1^{m+1}(-c)$
($c>0$) as a complex Einstein hypersurface. Accordingly, akin to $Q^m$,
the complex hyperbolic quadric $Q^{m\ast}$ is endowed with two significant geometric structures: an almost product structure $A$ and a K\"ahler structure $J$. These structures are anticommutative, i.e. $AJ = -JA$. 
For more details about $Q^{m\ast}$, we refer the readers to \cite{BS-2022,K-S,M-R,Suh4,V-W}.

Let $\overline{M}$ be an almost Hermitian manifold with almost complex structure $J$, and $M$ a connected orientable real hypersurface of $\overline{M}$ with unit normal vector field $N$. One defines the 
{\it Reeb vector field} on $M$ by $\xi := -JN$. Then $M$ is called a 
{\it Hopf hypersurface} if the integral curves of $\xi$ are geodesics. In K\"ahler manifolds, a real hypersurface being Hopf is equivalent to that
its Reeb vector field $\xi$ is a principal vector field (cf. \cite{BBW}). 
Over the past four decades, Hopf hypersurfaces have been a central object in the study of real hypersurfaces of non-flat complex space forms and other almost Hermitian manifolds. 
We refer the readers to \cite{B,BS-2022,CR,KM,NR} and the references therein for more details about this aspect. 
Among the most important results in this area are the complete classifications, due to Kimura \cite{KM} for the 
complex projective space $\mathbb{C}P^m$ ($m\geq2$) and to Berndt \cite{B} for the complex hyperbolic space $\mathbb{C}H^m$ ($m\geq2$), of Hopf hypersurfaces with constant principal curvatures. Their results are fundamental and highly influential in this field.

In this paper, we investigate Hopf hypersurfaces in the complex hyperbolic quadric $Q^{m*}$.  For $m=2$, the complex hyperbolic quadric $Q^{2*}$
with Einstein constant $-1$ is holomorphically isometric to the K\"ahler surface
$\mathbb{H}^2\times\mathbb{H}^2$. In this setting, Gao-Ma-Yao \cite{GMY} classified the homogeneous hypersurfaces, the isoparametric hypersurfaces and the hypersurfaces of $\mathbb{H}^2\times\mathbb{H}^2$ with at most two distinct constant principal curvatures. The classification of
Hopf hypersurfaces of $\mathbb{H}^2\times\mathbb{H}^2$ with constant principal curvatures can be obtained by Theorem 1.3 of \cite{ZGHY}.  
It is worth noting that several related results on hypersurfaces in $\mathbb{S}^2\times\mathbb{S}^2$ (the compact dual of 
$\mathbb{H}^2\times\mathbb{H}^2$) have also been obtained, see \cite{GHMY, LVWY, Ur}. 
For $m\geq3$, some known results about real hypersurfaces in $Q^{m*}$ include, for example, the classification of real hypersurfaces with isometric Reeb flow by Suh \cite{Suh4} and the classification of contact hypersurfaces by Berndt and Suh \cite{BS-2022}. 
For more studies on Hopf hypersurfaces of $Q^{m*}$, we refer to, e.g.,
\cite{B-2023,BS-2022,K-S,Suh4} and the references therein. 
Nonetheless, a complete classification of Hopf hypersurfaces with constant principal curvatures remains an open challenge. This paper is motivated by the following fundamental problem:

\vskip 2mm

\noindent {\bf Problem} Classify all the Hopf hypersurfaces of $Q^{m*}$ ($m\geq3$) with constant principal curvatures.

\vskip 2mm

In almost Hermitian manifolds, the classification of Hopf hypersurfaces with constant principal curvatures is a topic of significant interest. 
While this problem has been resolved in the complex hyperbolic space, the next natural candidate among noncompact symmetric spaces is the complex hyperbolic quadric $Q^{m*}$. It should be emphasized that the assumption of Hopf hypersurfaces in  $Q^{m*}$ with constant principal curvature is neither overly strong nor restrictive. Indeed, even in the complex hyperbolic space, the condition of constant principal curvatures alone is not sufficient to achieve a complete classification. For example, under the additional assumption that the Reeb vector field has two nontrivial projections onto the principal curvature spaces, 
D\'{\i}az-Ramos and Dom\'{\i}nguez-V\'{a}zquez \cite{DD1}  
obtained a complete classification of real hypersurfaces in the complex hyperbolic space with constant principal curvatures.



The classification problem for Hopf hypersurfaces with constant principal curvatures in $Q^{m*}$ ($m\geq3$) is significantly more complicated than that of $\mathbb{C}P^m$ and $\mathbb{C}H^m$. By Lemma \ref{lemma:2.5}, such a real hypersurface must have either $\mathfrak{A}$-principal unit normal vector field or $\mathfrak{A}$-isotropic unit normal vector field
(see definition \ref{def:2.1}). Moreover, Theorem \ref{thm:2.7} shows that a Hopf hypersurface in $Q^{m*}$ $(m\ge3)$ with $\mathfrak{A}$-principal unit normal vector field is either an open part of Example \ref{ex:3.1} or Example \ref{ex:3.2} or Example \ref{ex:3.3}. Consequently, we only need to focus on the Hopf hypersurface with constant principal curvatures and $\mathfrak{A}$-isotropic unit normal vector field. 
A foundational result (Theorem \ref{thm:4.2w}) shows that a Hopf hypersurface in $Q^{m*}$ $(m\ge3)$ with constant principal curvatures is an
isoparametric hypersurface (definition see page 12 in subsection \ref{sect:4.1}), and its parallel hypersurfaces also have constant principal curvatures. This allows us to apply an important theorem of Ge-Tang \cite{GT}, from which it follows that any focal submanifold (if it exists) must be austere.
We also establish Cartan's formulas (see \eqref{eqn:ca2}, \eqref{eqn:call} and \eqref{eqn:call2g}) for the Hopf hypersurface of $Q^{m*}$ ($m\geq3$) with constant principal curvatures and $\mathfrak{A}$-isotropic unit normal vector field. 
Unlike in $\mathbb{C}H^m$, these formulas involve the almost product structure $A$, making the interplay between the shape operator $S$, the complex structure $J$ and the almost product structure $A$ highly nontrivial. 
Given the central role of the almost product structure $A$, 
we establish a useful lemma which states that $A$ can project an eigenspace $V_{\lambda}$ ($\lambda\neq\pm1$) onto the orthogonal complement of $V_{\lambda}\oplus JV_{\lambda}$ (see Lemma \ref{lemma:5.1}). It is worth noting that, due to the curvature nature of the ambient manifold, Hopf hypersurfaces in $Q^{m*}$ with constant principal curvatures show more diversity than those in $Q^m$ (see \cite{LTY}).

In this paper, we study the classification problem of Hopf hypersurfaces of $Q^{m*}$ ($m\geq3$) with constant principal curvatures. By restricting the number of distinct principal curvatures to be at most two, we obtain the following result.
\begin{theorem}\label{thm:1.1a}
Let $M$ be a Hopf hypersurface of $Q^{m*}$ ($m\geq3$) with at most two distinct constant principal curvatures. Then $M$ is an open part of Example \ref{ex:3.1}.
\end{theorem}

By restricting the number of distinct principal curvatures to be three or four,
we can determine the values and multiplicities of the principal curvatures, and we
obtain the following results.

\begin{theorem}\label{thm:1.1b}
Let $M$ be a Hopf hypersurface of $Q^{m*}$ ($m\geq3$) with three distinct constant principal curvatures. Then, one of the following three cases occurs:
\begin{enumerate}
\item[(1)]
$M$ is either an open part of Example \ref{ex:3.2}, or Example \ref{ex:3.3}, or Example \ref{ex:3.4};

\item[(2)]
$M$ has constant principal curvatures $0,1,-1$ with multiplicities $3, m-2,m-2$, respectively.
In particular, Example \ref{ex:3.6} is contained in this case;

\item[(3)]
$M$ has constant principal curvatures $0,1,-1$ with multiplicities $2, m-1,m-2$, respectively.
In particular, Example \ref{ex:3.7} is contained in this case.
\end{enumerate}
\end{theorem}

\begin{theorem}\label{thm:1.1c}
Let $M$ be a Hopf hypersurface of $Q^{m*}$ ($m\geq3$) with four distinct constant principal curvatures. Then, one of the following four cases occurs:
\begin{enumerate}
\item[(1)]
$M$ is an open part of Example \ref{ex:3.5};

\item[(2)]
$M$ has constant principal curvatures $2\tanh(2r),0,1,-1$ with multiplicities $1,2, m-2,m-2$  respectively, where $r\in \mathbb{R}_+$ and $r\neq\frac{{\rm arctanh}(\frac{1}{2})}{2}$.
In particular, Example \ref{ex:3.8} is contained in this case;

\item[(3)]
$M$ has constant principal curvatures $2\coth(2r),0,1,-1$ with multiplicities $1,2, m-2,m-2$ respectively, where $r\in \mathbb{R}_+$.
In particular, Example \ref{ex:3.9} is contained in this case;

\item[(4)]
$M$ has constant principal curvatures $2,0,1,-1$ with multiplicities $1,2, m-2,m-2$, respectively.
In particular, Example \ref{ex:3.10} is contained in this case.
\end{enumerate}
\end{theorem}

\begin{remark}\label{rem:1.1}
In Theorems \ref{thm:1.1a}--\ref{thm:1.1c}, Examples \ref{ex:3.1}--\ref{ex:3.10} are all homogeneous real hypersurfaces of $Q^{m*}$ ($m\geq3$). Recall that,
Berndt and Tamaru (\cite{BT2}, for $m=3$), Berndt and Dom\'{\i}nguez-V\'{a}zquez (\cite{BD1}, for $m>3$) classified the cohomogeneity one actions on 
$Q^{m*}$, $m\geq3$, up to orbit equivalence. This naturally leads to the classification of homogeneous real hypersurfaces in $Q^{m*}$. Geometrically,
one can divide the homogeneous real hypersurfaces of $Q^{m*}$ into two types
according to whether there is a focal submanifold. The homogeneous hypersurfaces without focal submanifold were classified by Berndt and Tamaru \cite{BT1} in the more general context of irreducible Riemannian symmetric
spaces of noncompact type. The structural understanding of 
cohomogeneity one actions on symmetric spaces of noncompact type 
has been advanced by D\'{\i}az-Ramos, Dom\'{\i}nguez-V\'{a}zquez and Otero \cite{DDO}, and the classification of isometric cohomogeneity one actions up to orbit equivalence was recently completed by Sanmart\'{\i}n-L\'{o}pez and Solonenko \cite{SS}. Although homogeneous hypersurfaces in Riemannian manifolds are isoparametric hypersurfaces, the converse fails in general. For example, in noncompact symmetric spaces of rank $1$ and those of rank at least $3$,
uncountably many inhomogeneous isoparametric hypersurfaces have been constructed \cite{DD2,DDR,DS}. But so far, the existence of 
inhomogeneous isoparametric hypersurfaces in noncompact symmetric spaces of rank $2$ remains unknown. 
\end{remark}

\begin{remark}\label{rem:1.1aaa}
There are non-Hopf homogeneous real hypersurfaces of $Q^{m*}$ ($m\geq3$).
In fact, as pointed out in subsection 7.4 of \cite{BS-2022}, for any $t\neq(0,\frac{\pi}{4})$, horosphere $M_t$ is non-Hopf. Furthermore, $M_{\arctan(\frac{1}{2})}$ has four distinct constant principal curvatures. When $t\in(0,\arctan(\frac{1}{2}))\cup(\arctan(\frac{1}{2}),\frac{\pi}{4})$, $M_t$ has five distinct constant principal curvatures.
\end{remark}

\begin{remark}\label{rem:1.2}
For any real hypersurface in a K\"{a}hler manifold
with induced almost contact metric structure $(\phi,\xi,\eta,g)$ and the shape operator $S$, the
condition of $S\phi+\phi S=0$ characterizes the Hopf hypersurfaces for which
the maximal complex subbundle $\mathcal{C}$ is integrable (see Proposition 2.2 of \cite{B-2023}).
In non-flat complex space forms, it is quite remarkable and not obvious that
there are no Hopf hypersurfaces whose maximal complex subbundle $\mathcal{C}$ is integrable.
From the investigations in \cite{BS-2022}, there are also no Hopf
hypersurfaces in the complex quadric $Q^m$ for which $\mathcal{C}$ is integrable.
It is interesting that Berndt \cite{B-2023} constructed explicitly a one-parameter family of homogeneous Hopf hypersurfaces in $Q^{m*}$ for which $\mathcal{C}$ 
is integrable. These homogeneous Hopf hypersurfaces are exactly the Examples \ref{ex:3.6}--\ref{ex:3.10}.
These are also the first known examples which satisfy the condition of $S\phi+\phi S=0$ in irreducible K\"ahler manifolds. Despite this progress, the problem of classifying all real hypersurfaces in $Q^{m*}$ ($m\geq3$) that satisfy 
$S\phi+\phi S=0$ remains open.
\end{remark}

The paper is organized as follows. In Section \ref{sect:2}, we collect some basic
properties of $Q^{m*}$ ($m\geq3$) and some preliminaries of the geometry of the real hypersurfaces of $Q^{m*}$.
In Section \ref{sect:3}, we introduce some known examples of Hopf hypersurfaces with constant principal curvatures in $Q^{m*}$, which are all homogeneous real hypersurfaces. In Section \ref{sect:4}, we study the parallel hypersurfaces and the focal submanifold (if it exists) of Hopf hypersurfaces of $Q^{m*}$ ($m\geq3$) with constant principal curvatures and $\mathfrak{A}$-isotropic unit normal vector field, we also establish Cartan's formulas for Hopf hypersurfaces with constant principal curvatures and $\mathfrak{A}$-isotropic unit normal vector field. Finally, Section \ref{sect:5} is dedicated to
the proofs of Theorems \ref{thm:1.1a},  \ref{thm:1.1b} and \ref{thm:1.1c}.

\textbf{Acknowledgments:}
H. Li was supported by NSFC Grant No.12471047. H. Tamaru was supported by JSPS KAKENHI Grant Numbers JP22H01124 and JP24K21193, and was also partly supported by MEXT Promotion of Distinctive Joint
Research Center Program JPMXP0723833165. Z. Yao was supported by NSFC Grant No. 12401061.

\section{Preliminaries}\label{sect:2}

\subsection{The complex hyperbolic quadric}\label{sect:2.1}~


As we have described in last section, the complex hyperbolic quadric
$Q^{m\ast}$ can be isometrically
immersed into the complex anti-de Sitter space $\mathbb{C}H_1^{m+1}(-c)$
($c>0$) as a complex Einstein hypersurface.
The pseudo-Euclidean space $\mathbb{R}^{n+1}_{i+1}$ is $\mathbb{R}^{n+1}$ equipped with the metric
$$
\langle(x_1,\ldots, x_{n+1}),(y_1,\ldots, y_{n+1})\rangle_{i+1}=-x_1y_1-\cdots-x_{i+1}y_{i+1}
+x_{i+2}y_{i+2}+\cdots+x_{n+1}y_{n+1}.
$$
The $n$-dimensional pseudo-hyperbolic space $\mathbb{H}^n_i(-1)$ with  constant sectional curvature $-1$ is
$$
\mathbb{H}^n_i(-1)=\{x\in \mathbb{R}^{n+1}_{i+1}|\langle x,x\rangle_{i+1}=-1\}.
$$

We consider the complex space $\mathbb{C}^{m+2}_2$ with the metric $\langle\langle(z_1,\ldots, z_{m+2}),(w_1,\ldots, w_{m+2})\rangle\rangle_2={\rm Re}(-z_1\bar{w}_1-z_2\bar{ w}_2+\cdots+z_{m+2}\bar{w}_{m+2})$.
Define the complex anti-de Sitter space $\mathbb{C}H^{m+1}_1$ as the set of all complex $1$-dimensional subspaces of $\mathbb{C}^{m+2}_2$. Under the identification $\mathbb{R}^{2m+4}_4\cong \mathbb{C}^{m+2}_2$, we refer to the map
$$
\pi:\mathbb{H}^{2m+3}_3(-1)\rightarrow \mathbb{C}H^{m+1}_1,\ \ \ z\mapsto[z]
$$
as the Hopf fibration. If we equip $\mathbb{C}H^{m+1}_1$ with a pseudo-Riemannian metric $g$ such that $\pi$
becomes a pseudo-Riemannian submersion, then $(\mathbb{C}H^{m+1}_1,g)$ 
has complex index $1$ and constant holomorphic sectional
curvature $-4$, we write  $\mathbb{C}H^{m+1}_1(-4)=(\mathbb{C}H^{m+1}_1,g)$.

For any
$z\in \mathbb{H}^{2m+3}_3(-1)$, we have that $\pi^{-1}([z])=\{e^{\mathbf{i}t}z|t\in \mathbb{R}\}$ and
${\rm ker}(d\pi)_z={\rm span}\{\mathbf{i}z\}$. The complex structure $J$
on $\mathbb{C}H^{m+1}_1(-4)$ is induced by multiplication by $\mathbf{i}$ on $T\mathbb{H}^{2m+3}_3(-1)$ and
$(\mathbb{C}H^{m+1}_1(-4), g, J)$ is a K\"{a}hler manifold.

The complex hyperbolic quadric $Q^{m*}$ as the complex hypersurface of $\mathbb{C}H^{m+1}_1(-4)$ is defined by
$$
Q^{m*}=\{[(z_1,z_1,\ldots,z_{m+2})]\in \mathbb{C}H^{m+1}_1(-4)|-z_1^2-z_2^2+\cdots+z_{m+2}^2=0\}.
$$
$Q^{m*}$ is equipped with the induced metric, which we still denote by $g$, and the induced complex structure, which we still denote by $J$, then $(Q^{m*}, g, J)$ is a K\"{a}hler manifold.

The inverse image of $Q^{m*}$ under the Hopf fibration is given by
$$
V^{2m+1*}_1=\{u+\mathbf{i}v|u,v\in \mathbb{R}^{m+2}_2,\langle u,u\rangle_2=\langle v,v\rangle_2=-\frac{1}{2},
\langle u,v\rangle_2=0\}\subset \mathbb{H}^{2m+3}_3(-1).
$$
Then $Q^{m*}$ is a Riemannian submanifold of $\mathbb{C}H^{m+1}_1(-4)$, where the normal space
$T^\perp_{[z]}Q^{m*}$ is spanned by $(d\pi)_z(\bar{z})$ and $J(d\pi)_z(\bar{z})=(d\pi)_z(\mathbf{i}\bar{z})$.


We denote by $\mathfrak{A}$ the set of all shape operators of $Q^{m*}$ in $\mathbb{C}H^{m+1}_1(-4)$  associated with unit normal vector fields, and $\mathfrak{A}$ is a collection of $(1,1)$-tensor fields on $Q^{m*}$.
Then, $\mathfrak{A}$ is an $S^1$-subbundle of the endomorphism
bundle End($TQ^{m*}$). In summary, we have

\begin{lemma}[cf. \cite{BS-2022,K-S,M-R,V-W}]\label{lemma:2.1}
For each $A\in\mathfrak{A}$, it holds that
\begin{equation}\label{eqn:2.1}
A^2=Id, \ \ g(AX,Y)=g(X,AY), \ \ AJ=-JA, \ \ \forall X,Y\in TQ^{m*}.
\end{equation}
\end{lemma}

Lemma \ref{lemma:2.1} implies in
particular that $\mathfrak{A}$ is a family of {\it almost product structures}
on $Q^{m*}$ and we will use this term in this paper.
Furthermore, for each almost product structure $A\in\mathfrak{A}$, there exists a
one-form $q$ on $Q^{m*}$ such that it holds the relation
\begin{equation}\label{eqn:2.2}
(\bar{\nabla}_XA)Y=q(X)JAY,\ \ \forall\, X,Y\in TQ^{m*},
\end{equation}
where $\bar{\nabla}$ denotes the Levi-Civita connection of $Q^{m*}$.

The Gauss equation for the complex hypersurface $Q^{m*}\hookrightarrow \mathbb{C}H^{m+1}_1(-4)$
implies that the Riemannian curvature tensor $\bar{R}$ of $Q^{m*}$ can be expressed in terms
of the Riemannian metric $g$, the complex structure $J$ and a generic almost product
structure $A\in\mathfrak{A}$ as follows:
\begin{equation}\label{eqn:2.3}
\begin{aligned}
\bar{R}(X,Y)Z=&-g(Y,Z)X+g(X,Z)Y\\
&-g(JY,Z)JX+g(JX,Z)JY+2g(JX,Y)JZ\\
&-g(AY,Z)AX+g(AX,Z)AY\\
&-g(JAY,Z)JAX+g(JAX,Z)JAY,
\end{aligned}
\end{equation}
where, it should be noted that $\bar{R}$ is independent of the special choice of $A\in\mathfrak{A}$.

According to \cite{BS-2022}, there are two types of singular tangent vectors for $Q^{m*}$, namely,
$\mathfrak{A}$-principal and $\mathfrak{A}$-isotropic tangent vectors, described as
follows: 
\begin{enumerate}[1.]
\item
If there exists an almost product structure $A\in \mathfrak{A}_{[z]}$ such that $W\in V(A)$,
then $W$ is singular. Here, we denote $V(A)$ as
the eigenspace of $A$ corresponding to $1$. Such $W\in T_{[z]}Q^{m*}$ is called an
$\mathfrak{A}$-principal tangent vector.

\item
If there exists an almost product structure $A\in \mathfrak{A}_{[z]}$ and orthonormal
vectors $X,Y\in V(A)$ such that $\tfrac{W}{\|W\|}=\tfrac{X+JY}{\sqrt{2}}$, then $W$ is
singular. In this case we have $g(AW,W)=0$ and $W\in T_{[z]}Q^{m*}$ is
called an $\mathfrak{A}$-isotropic tangent vector.
\end{enumerate}
For some further results about the complex hyperbolic quadric, we refer the readers to \cite{BS-2022,K-S,M-R,V-W}.

\subsection{Real hypersurfaces of the complex hyperbolic quadric}\label{sect:2.2}~

Let $M$ be a real hypersurface of $Q^{m*}$ and $N$ its unit normal vector field. For any
tangent vector field $X$ of $M$, we have the decomposition
\begin{equation}\label{eqn:2.4}
JX=\phi X+ \eta(X)N,
\end{equation}
where $\phi X$ and $\eta(X)N$ are, respectively, the tangent and normal parts of
$JX$. Then $\phi$ is a tensor field of type $(1,1)$, and $\eta$ is a $1$-form on $M$.
By definition, $\phi$ and $\eta$ satisfy the following relations:
\begin{equation}\label{eqn:2.5}
\left\{
\begin{aligned}
&\eta(X)=g(X,\xi),\ \ \eta(\phi X)=0,\ \ \phi^2X=-X+\eta(X)\xi,\\
&g(\phi X,Y)=-g(X,\phi Y),\ \ g(\phi X,\phi Y)=g(X,Y)-\eta(X)\eta(Y),
\end{aligned}\right.
\end{equation}
where $\xi:=-JN$ is called the {\it Reeb vector field} of $M$. The equations in
\eqref{eqn:2.5} show that $\{\phi,\xi,\eta\}$ determines an {\it almost contact
structure} on $M$.

Let $\nabla$ be the induced connection on $M$ with $R$ its Riemannian curvature
tensor. The formulas of Gauss and Weingarten state that
\begin{equation}\label{eqn:2.6}
\begin{split}
\bar \nabla_X Y=\nabla_X Y + g(SX,Y)N,\quad \bar \nabla_X N=-
S X , \ \ \forall\, X,Y \in TM,
\end{split}
\end{equation}
where $S$ is the shape operator of $M\hookrightarrow Q^{m*}$. Let $H={\rm Tr}S$
be the mean curvature of $M$. By using \eqref{eqn:2.6},
we can derive that
\begin{equation}\label{eqn:2.7}
\nabla_X \xi=\phi SX.
\end{equation}

The Gauss and Codazzi equations of $M$ are given by
\begin{equation}\label{eqn:2.8}
\begin{split}
R(X,Y)Z=&-g(Y,Z)X+g(X,Z)Y-g(\phi Y,Z)\phi X+g(\phi X,Z)\phi Y\\
 &+2g(\phi X,Y)\phi Z-g(AY,Z)(AX)^\top+g(AX,Z)(AY)^\top\\
 &-g(JAY,Z)(JAX)^\top+g(JAX,Z)(JAY)^\top\\
 &+g(SZ,Y)SX-g(SZ,X)SY,
 \end{split}
\end{equation}
and
\begin{equation}\label{eqn:2.9}
\begin{split}
(\nabla_X S)Y-(\nabla_Y S)X=&-\eta(X)\phi Y+\eta(Y)\phi X+2g(\phi X,Y)\xi\\
&-g(X,AN)(AY)^\top+g(Y,AN)(AX)^\top\\
 &\qquad -g(X,A\xi)(JAY)^\top +g(Y,A\xi)(JAX)^\top,
\end{split}
\end{equation}
where $\cdot^\top$ means the tangential part.

%

The tangent bundle $TM$ of $M$ splits orthogonally into
$TM=\mathcal{C}\oplus \mathcal{F}$, where $\mathcal{C}=\text{ker}(\eta)$
is the maximal complex subbundle of $TM$ and $\mathcal{F}=\mathbb{R}\xi$.
When restricted to $\mathcal{C}$, the structure tensor field $\phi$
coincides with the complex structure $J$. Moreover, at each point
$[z]\in M$, the set
$$
\mathcal{Q}_{[z]}=\{X\in T_{[z]}M\mid AX\in T_{[z]}M\ \text{for all}\ A\in \mathfrak{A}_{[z]}\}
$$
defines a maximal $\mathfrak{A}_{[z]}$-invariant subspace of $T_{[z]}M$.

The following definitions can be found in Berndt and Suh \cite{BS-2022,Suh4}.

\begin{definition}[\cite{BS-2022,Suh4}]\label{def:2.1}
Let $M$ be a real hypersurface of $Q^{m*}$ ($m\geq3$) with unit normal vector field $N$.

{\rm (a)} If there exists an almost product structure $A\in \mathfrak{A}$
such that $AN=N$ everywhere, then we say that $N$ is $\mathfrak{A}$-principal on $M$,
and $M$ has $\mathfrak{A}$-principal unit normal vector field $N$.

{\rm (b)} If there exists an almost product structure $A\in\mathfrak{A}$
such that $AN,A\xi\in \mathcal{C}$ everywhere, then we say that
$N$ is $\mathfrak{A}$-isotropic on $M$, and $M$ has $\mathfrak{A}$-isotropic unit normal vector field $N$.
\end{definition}

\begin{remark}\label{rem:2.1}
If there exists an almost product structure $A_0\in\mathfrak{A}$
such that $A_0N,A_0\xi\in \mathcal{C}$ everywhere, then for any almost product structure $A\in\mathfrak{A}$,
it also holds $AN,A\xi\in \mathcal{C}$ everywhere.
\end{remark}

\vskip 1mm
In later sections, we also need the following lemmas on Hopf hypersurfaces in $Q^{m*}$.
Assume that $M$ is a Hopf hypersurface of $Q^{m*}$ with Reeb vector field $\xi$ satisfying
$\nabla_\xi\xi=0$ and $S\xi=\alpha \xi$, here the function $\alpha$ is called the Reeb
function. Then, we have

\begin{lemma}[\cite{BS-2022}]\label{lemma:2.3}
Let $M$ be a Hopf hypersurface in $Q^{m*}$ with unit normal vector field
$N$ and $S\xi=\alpha\xi$. 
%
%
Then for any vector fields $X,Y\in \mathcal{Q}$, it holds
\begin{equation}\label{eqn:2.13}
\begin{aligned}
2g(S\phi SX,Y)-\alpha g((\phi S+S\phi)X,Y)+2g(\phi X,Y)=0.
\end{aligned}
\end{equation}
\end{lemma}

\begin{proof}
By using the fact $g(AN,X)=g(A\xi,X)=0$ for any $X\in \mathcal{Q}$, 
\eqref{eqn:2.13} can be obtained directly from the Lemma 7.6.1 in \cite{BS-2022}.
\end{proof}

\begin{lemma}[Lemma 3.4 of \cite{WLS}]\label{lemma:2.5}
Let $M$ be a Hopf hypersurface in $Q^{m*}$ ($m\geq3$). If the Reeb function $\alpha$ is constant,
then $M$ has either $\mathfrak{A}$-principal unit normal vector field or $\mathfrak{A}$-isotropic unit normal vector field.
\end{lemma}



For real hypersurfaces in $Q^{m*}\ (m\ge3)$ with $\mathfrak{A}$-principal unit normal vector field or $\mathfrak{A}$-isotropic unit normal vector field, according the properties of the almost product structure $A$ and the complex structure $J$, we have the following results.

\begin{lemma}\label{lemma:2.6}
Let $M$ be a real hypersurface in $Q^{m*}\ (m\ge3)$ with $\mathfrak{A}$-isotropic
unit normal vector field $N$. Then there exists an almost product structure $A\in\mathfrak{A}$
on $M$ such that $SAN=SA\xi=0$.
\end{lemma}
\begin{proof}
The proof is totally similar to the situation of the real hypersurface in $Q^m$ with $\mathfrak{A}$-isotropic unit normal vector field (see Lemma 3.5 of \cite{Loo11}).
Since $M$ has $\mathfrak{A}$-isotropic unit normal vector field $N$, there exists
an almost product structure $A\in \mathfrak{A}$, such that $AN,A\xi\in \mathcal{C}$ everywhere,
so $\{N, \xi, AN, A\xi\}$ are
orthogonal to each other and $\phi AN=JAN=A\xi$. Then, for any $X\in TM$, from \eqref{eqn:2.2}, \eqref{eqn:2.6} and \eqref{eqn:2.7}, we have
$$
\begin{aligned}
0&=X[g(AN,N)]=g(\bar{\nabla}_X(AN),N)+g(AN,\bar{\nabla}_XN)\\
&=g(q(X)A\xi-ASX,N)-g(AN,SX)=-2g(SAN,X).
\end{aligned}
$$
It follows that $SAN=0$. 
Similarly, by direct calculation and using \eqref{eqn:2.2}, \eqref{eqn:2.6} and \eqref{eqn:2.7},
we obtain
$$
\begin{aligned}
0&=X[g(A\xi,N)]=g(\bar{\nabla}_X(A\xi),N)+g(A\xi,\bar{\nabla}_XN)\\
&=g(q(X)JA\xi+A\phi SX+g(SX,\xi)AN,N)-g(SA\xi,X)\\
&=-2g(SA\xi,X),\ \ \forall\, X\in TM.
\end{aligned}
$$
It follows that $SA\xi=0$.
\end{proof}

\begin{remark}\label{rem:2.2aaas}
From the fact that $\mathfrak{A}$ is an $S^1$-subbundle of the endomorphism
bundle End($TQ^{m*}$), then Lemma \ref{lemma:2.6} can also imply that, on a
real hypersurface $M$ of $Q^{m*}$ ($m\geq3$) with $\mathfrak{A}$-isotropic unit normal vector field $N$,
it holds that $SAN=SA\xi=0$ for any almost product structure $A\in\mathfrak{A}$.
\end{remark}

\begin{theorem}[Proposition 4.4 of \cite{KSW}]\label{thm:2.7}
Let $M$ be a Hopf hypersurface of $Q^{m*}$ ($m\geq3$) with $\mathfrak{A}$-principal unit normal vector field $N$. Then,
$M$ is either an open part of Example \ref{ex:3.1} or Example \ref{ex:3.2} or Example \ref{ex:3.3}.
\end{theorem}

\begin{theorem}[Theorem 1.1 of \cite{Suh4}]\label{thm:2.2}
Let $M$ be a real hypersurface of $Q^{m*}$ ($m\geq3$) with isometric Reeb flow.
Then, $M$ is either an open part of Example \ref{ex:3.4} or
Example \ref{ex:3.5}.
\end{theorem}

\section{Canonical examples of Hopf hypersurfaces with constant principal curvatures}\label{sect:3}

In this section, we briefly give the known homogeneous real hypersurfaces of $Q^{m*}$,
which are also Hopf hypersurfaces. Some examples are constructed by
using the fact that the complex hyperbolic quadric $Q^{m*}$ is a Hermitian symmetric space,
and $Q^{m*}$ is realized as the quotient manifold $SO^o_{2,m}/(SO_2\times SO_m)$.
For more details about the construction of the complex hyperbolic quadric $Q^{m*}$ as a Riemannian
symmetric space, we refer to \cite{B-2023,BS-2022,K-S}.

First of all, there are three types of homogeneous real hypersurfaces of $Q^{m*}$ with $\mathfrak{A}$-principal unit normal vector field. These examples were first presented in \cite{BS-2015}
and then studied in further detail in \cite{BS-2022,K-S}.

\begin{example}[Theorem 7.4.4 of \cite{BS-2022}]\label{ex:3.1}
Let $M$ be a horosphere in $Q^{m*}$ with its center at infinity being given by an
$\mathfrak{A}$-principal geodesic $\gamma$. Then the following statements hold:
\begin{enumerate}
\item[(1)]
$M$ is a Hopf hypersurface of $Q^{m*}$ with $\mathfrak{A}$-principal unit normal vector field,
its Reeb function is $\alpha=\sqrt{2}$.

\item[(2)]
$M$ has two distinct constant principal curvatures:
$$
\begin{tabular}{|c|c|c|}
	\hline
	{\rm value} & $\sqrt{2}$ & $0$  \\
	\hline
	{\rm multiplicity} & $m$ & $m-1$ \\
	\hline
\end{tabular}
$$
\end{enumerate}
\end{example}

\begin{example}[Theorem 7.5.1 of \cite{BS-2022}]\label{ex:3.2}
Let $M$ be the tube with radius $r\in \mathbb{R}_+$ around the totally geodesic $Q^{m-1*}$ in $Q^{m*}$. For $M$ the following statements hold:
\begin{enumerate}
\item[(1)]
$M$ is a Hopf hypersurface of $Q^{m*}$ with $\mathfrak{A}$-principal unit normal vector field,
its Reeb function is $\alpha=\sqrt{2}\coth(\sqrt{2}r)$.

\item[(2)]
$M$ has three distinct constant principal curvatures:
$$
\begin{tabular}{|c|c|c|c|}
	\hline
	{\rm value} & $\sqrt{2}\coth(\sqrt{2}r)$ & $0$ & $\sqrt{2}\tanh(\sqrt{2}r)$  \\
	\hline
	{\rm multiplicity} & $1$ & $m-1$ & $m-1$ \\
	\hline
\end{tabular}
$$
\end{enumerate}
\end{example}

\begin{example}[Theorem 7.5.3 of \cite{BS-2022}]\label{ex:3.3}
Let $M$ be the tube with radius $r\in \mathbb{R}_+$ around the totally geodesic $\mathbb{R}H^m$ in $Q^{m*}$. For $M$ the following statements hold:
\begin{enumerate}
\item[(1)]
$M$ is a Hopf hypersurface of $Q^{m*}$ with $\mathfrak{A}$-principal unit normal vector field,
its Reeb function is $\alpha=\sqrt{2}\tanh(\sqrt{2}r)$.

\item[(2)]
$M$ has three distinct constant principal curvatures:
$$
\begin{tabular}{|c|c|c|c|}
	\hline
	{\rm value} & $\sqrt{2}\tanh(\sqrt{2}r)$ & $0$ & $\sqrt{2}\coth(\sqrt{2}r)$  \\
	\hline
	{\rm multiplicity} & $1$ & $m-1$ & $m-1$ \\
	\hline
\end{tabular}
$$
\end{enumerate}
\end{example}

Next, the following two types of homogeneous real hypersurfaces of $Q^{m*}$ with $\mathfrak{A}$-isotropic unit normal vector field were first presented in \cite{Suh4}
and then studied in further detail in \cite{BS-2022}. These examples are characterized as those real hypersurfaces with isometric Reeb flow.

\begin{example}[Theorem 7.4.2 of \cite{BS-2022}]\label{ex:3.4}
Let $M$ be the horosphere in $Q^{m*}$ whose center at infinity is in the equivalent class
of an $\mathfrak{A}$-isotropic singular geodesic in $Q^{m*}$. Then the following statements hold
\begin{enumerate}
\item[(1)]
$M$ is a Hopf hypersurface of $Q^{m*}$ with $\mathfrak{A}$-isotropic unit normal vector field,
its Reeb function is $\alpha=2$.

\item[(2)]
$M$ has three distinct constant  principal curvatures:
$$
\begin{tabular}{|c|c|c|c|}
	\hline
	{\rm value} & $2$ & $0$ & $1$  \\
	\hline
	{\rm multiplicity} & $1$ & $2$ & $2m-4$ \\
	\hline
\end{tabular}
$$
\end{enumerate}
\end{example}

\begin{example}[Theorem 7.5.5 of \cite{BS-2022}]\label{ex:3.5}
Let $M$ be the tube with radius $r\in \mathbb{R}_+$ around the totally geodesic $\mathbb{C}H^k$ in $Q^{2k*}$,
$k\geq2$. Then the following statements hold
\begin{enumerate}
\item[(1)]
$M$ is a Hopf hypersurface of $Q^{2k*}$ with $\mathfrak{A}$-isotropic unit normal vector field,
its Reeb function is $\alpha=2\coth(2r)$.

\item[(2)]
It has four distinct constant  principal curvatures:
$$
\begin{tabular}{|c|c|c|c|c|}
	\hline
	{\rm value} & $2\coth(2r)$ & $0$ & $\tanh(r)$  & $\coth(r)$\\
	\hline
	{\rm multiplicity} & $1$ & $2$ & $2k-2$ & $2k-2$\\
	\hline
\end{tabular}
$$
\end{enumerate}
\end{example}

Finally, Berndt \cite{B-2023} recently constructed a one-parameter family
$M_{\alpha}^{2m-1}$, $0\leq\alpha<+\infty$, of (pairwise noncongruent)
homogeneous Hopf hypersurfaces in $Q^{m*}$ ($m\geq3$), whose maximal complex subbundle of the
tangent bundle is integrable. These real hypersurfaces are the first known examples which satisfy the condition of $S\phi+\phi S=0$ in irreducible K\"ahler manifolds. Here we briefly describe the properties of these real hypersurfaces. For more details, see \cite{B-2023}.

\begin{example}[Theorem 6.1 of \cite{B-2023}]\label{ex:3.6}
Let $M^{2m-1}_0$ be the homogeneous real hypersurface in $Q^{m*}$ obtained by
canonical extension of the geodesic that is tangent to the root vector $H_{\alpha_1}$ in the boundary
component $B_1\cong \mathbb{C}H^1(-4)$ of $Q^{m*}$. Then the following statements hold
\begin{enumerate}
\item[(1)]
$M^{2m-1}_0$ is a Hopf hypersurface of $Q^{m*}$ with $\mathfrak{A}$-isotropic unit normal vector field,
its Reeb function is $\alpha=0$.

\item[(2)]
$M^{2m-1}_0$ has three distinct constant  principal curvatures:
$$
\begin{tabular}{|c|c|c|c|}
	\hline
	{\rm value} & $0$ & $-1$  & $1$\\
	\hline
	{\rm multiplicity} & $3$ & $m-2$ & $m-2$\\
	\hline
\end{tabular}
$$
\end{enumerate}
\end{example}

\begin{example}[Theorem 7.1 of \cite{B-2023}]\label{ex:3.7}
Let $M^{2m-1}_0$ be the homogeneous real hypersurface in $Q^{m*}$ obtained by
canonical extension of the geodesic that is tangent to the root vector $H_{\alpha_1}$ in the boundary
component $B_1\cong \mathbb{C}H^1(-4)$ of $Q^{m*}$.
Let $M^{2m-1}_1$ be the equidistant real hypersurface at oriented distance $r=\frac{{\rm arctanh}(\frac{1}{2})}{2}$ from
$M^{2m-1}_0$. Then the following statements hold
\begin{enumerate}
\item[(1)]
$M^{2m-1}_1$ is a Hopf hypersurface of $Q^{m*}$ with $\mathfrak{A}$-isotropic unit normal vector field, its Reeb function is $\alpha=1$.

\item[(2)]
$M^{2m-1}_1$ has three distinct constant  principal curvatures:
$$
\begin{tabular}{|c|c|c|c|}
	\hline
	{\rm value} & $1$ & $0$  & $-1$\\
	\hline
	{\rm multiplicity} & $m-1$ & $2$ & $m-2$\\
	\hline
\end{tabular}
$$
\end{enumerate}
\end{example}

\begin{example}[Theorem 7.1 of \cite{B-2023}]\label{ex:3.8}
Let $M^{2m-1}_0$ be the homogeneous real hypersurface in $Q^{m*}$ obtained by
canonical extension of the geodesic that is tangent to the root vector $H_{\alpha_1}$ in the boundary
component $B_1\cong \mathbb{C}H^1(-4)$ of $Q^{m*}$.
Let $M^{2m-1}_{2\tanh(2r)}$ be the equidistant real hypersurface at oriented distance $r\neq\frac{{\rm arctanh}(\frac{1}{2})}{2}$ from
$M^{2m-1}_0$. Then the following statements hold
\begin{enumerate}
\item[(1)]
$M^{2m-1}_{2\tanh(2r)}$ is a Hopf hypersurface of $Q^{m*}$ with $\mathfrak{A}$-isotropic unit normal vector field, its Reeb function is $\alpha=2\tanh(2r)$.

\item[(2)]
$M^{2m-1}_{2\tanh(2r)}$ has four distinct constant  principal curvatures:
$$
\begin{tabular}{|c|c|c|c|c|}
	\hline
	{\rm value} & $2\tanh(2r)$ & $0$  & $-1$ & $1$\\
	\hline
	{\rm multiplicity} & $1$ & $2$ & $m-2$ & $m-2$\\
	\hline
\end{tabular}
$$
\end{enumerate}
\end{example}

In \cite{B-2023}, by extending a point in the boundary
component $B_1\cong \mathbb{C}H^1(-4)$, Berndt obtained an isometric embedding $P^{n-1}$ of the complex hyperbolic space $\mathbb{C}H^{m-1}(-4)$ into $Q^{m*}$ as a homogeneous complex hypersurface.
Then, the tubes of $P^{n-1}$ are homogeneous real hypersurfaces in $Q^{m*}$.

\begin{example}[Theorem 5.2 of \cite{B-2023}]\label{ex:3.9}
Let $P^{2m-1}_r$ be the tube with radius $r\in \mathbb{R}_+$ around the homogeneous complex
hypersurface $P^{m-1}\cong \mathbb{C}H^{m-1}(-4)$ in $Q^{m*}$. Then the following statements hold:
\begin{enumerate}
\item[(1)]
$P^{2m-1}_r$ is a Hopf hypersurface of $Q^{m*}$ with $\mathfrak{A}$-isotropic unit normal vector field, its Reeb function is $\alpha=2\coth(2r)$.

\item[(2)]
$P^{2m-1}_r$ has four distinct constant principal curvatures: 
$$
\begin{tabular}{|c|c|c|c|c|}
	\hline
	{\rm value} & $2\coth(2r)$ & $0$  & $-1$ & $1$\\
	\hline
	{\rm multiplicity} & $1$ & $2$ & $m-2$ & $m-2$\\
	\hline
\end{tabular}
$$
\end{enumerate}
\end{example}

In \cite{B-2023}, by using the canonical extension of a horocycle in the boundary component
$B_1\cong \mathbb{C}H^1(-4)$, Berndt obtained the homogeneous real hypersurface $M^{2m-1}_2$.

\begin{example}[Theorem 8.1 of \cite{B-2023}]\label{ex:3.10}
The homogeneous real hypersurface $M^{2m-1}_2$ in $Q^{m*}$ satisfies the following properties:
\begin{enumerate}
\item[(1)]
$M^{2m-1}_2$ is a Hopf hypersurface of $Q^{m*}$ with $\mathfrak{A}$-isotropic unit normal vector field, its Reeb function is $\alpha=2$.

\item[(2)]
$M^{2m-1}_2$ has four distinct constant  principal curvatures:
$$
\begin{tabular}{|c|c|c|c|c|}
	\hline
	{\rm value} & $2$ & $0$  & $-1$ & $1$\\
	\hline
	{\rm multiplicity} & $1$ & $2$ & $m-2$ & $m-2$\\
	\hline
\end{tabular}
$$
\end{enumerate}
\end{example}

\begin{remark}\label{rem:3.1}
Examples \ref{ex:3.1}, \ref{ex:3.4}, \ref{ex:3.6}, \ref{ex:3.7}, \ref{ex:3.8} and \ref{ex:3.10}
all have no focal submanifold. The maximal complex
subbundles of the tangent bundles of Examples \ref{ex:3.6}--\ref{ex:3.10} are integrable, and their integral manifolds are all the homogeneous complex hypersurface $P^{n-1}$.
\end{remark}

\begin{remark}\label{rem:3.2}
The maximal $\mathfrak{A}$-invariant subspace $\mathcal{Q}$ of Examples \ref{ex:3.6}--\ref{ex:3.10} are all $S$-invariant, and the eigenvalues of $S$ restricted to $\mathcal{Q}$ are $1$ and $-1$. Let $V_1$ and $V_{-1}$ be the eigenspaces restricted to $\mathcal{Q}$ corresponding to $1$ and $-1$.
Considering the important role played by the almost product structure $A$ in the study of real hypersurfaces of $Q^{m*}$, it is necessary to understand the relationship between the almost product structure $A$ and the shape operator $S$ on Examples \ref{ex:3.6}--\ref{ex:3.10}.
According to \cite{B-2023}, on each one of these examples, there exists an almost product structure
$A\in\mathfrak{A}$ such that $AV_1=V_1$ and $AV_{-1}=V_{-1}$.
\end{remark}

\section{Parallel hypersurfaces, focal submanifolds and Cartan's formulas}\label{sect:4}


\subsection{Parallel hypersurfaces and focal submanifold of the Hopf hypersurface with constant principal curvatures and $\mathfrak{A}$-isotropic unit normal vector field}\label{sect:4.1}~

In this subsection, we study the parallel hypersurfaces and the focal submanifold of the Hopf hypersurface of $Q^{m*}$ ($m\geq3$) with constant principal curvatures and $\mathfrak{A}$-isotropic unit normal vector field.

Let $M$ be a Hopf hypersurface of $Q^{m*}$ ($m\geq3$) with constant principal curvatures and $\mathfrak{A}$-isotropic unit normal vector field. Then $S\xi=\alpha\xi$. Up to a sign of the unit normal vector field $N$, we can assume that $\alpha\geq0$. By $SAN=SA\xi=0$
and $\mathcal{Q}=TM\ominus{\rm Span}\{\xi,AN,A\xi\}$,
we know that $\mathcal{Q}$ is $S$-invariant, $J$-invariant and $A$-invariant. We denote the set
of eigenvalues of $S$ restricted to $\mathcal{Q}$ by $\sigma(\mathcal{Q})$.
For any $\lambda\in\sigma(\mathcal{Q})$, let $V_\lambda$ be the corresponding eigenspace
restricted to $\mathcal{Q}$.

We define the map
$$
\Phi_r: M\rightarrow Q^{m*},\ \ \ \ \ p \mapsto \Phi_r(p)=\exp_{p}(rN_{p}), \quad \Phi_{0}(M)=M,
$$
where $p\in M$ and ${\rm exp}$ is the Riemannian exponential map of $Q^{m*}$
and $N$ is the unit normal vector field of $M$. Here, we allow that $r$ can be negative. If $r<0$, we means that $\Phi_r(M)$ is obtained by moving a distance $|r|$ from $M$ along the unit normal vector field $-N$. Let $\{E_1,\ldots, E_{2m-4},E_{2m-3}=AN,E_{2m-2}=A\xi,E_{2m-1}=\xi\}$ be an orthonormal basis at $p\in M$
such that $SE_i=\lambda_iE_i$ and $E_i\in\mathcal{Q}$ for $1\leq i\leq 2m-4$.
Let $\{E_1^r,\ldots, E_{2m-4}^r,E_{2m-3}^r,E_{2m-2}^r,E_{2m-1}^r\}$
be the parallel translation of $\{E_i\}_{i=1}^{2m-1}$ along the geodesic to the nearby parallel hypersurface $\Phi_r(M)$.
Then applying the standard Jacobi field theory,
we have the following proposition about the parallel hypersurfaces of $M$.

\begin{proposition}\label{prop:4.1w}
Let $M$ be a Hopf hypersurface of $Q^{m*}$ ($m\geq3$) with constant principal curvatures and $\mathfrak{A}$-isotropic unit normal vector field $N$.
Then, for any $p\in M$, the tangent map of $\Phi_r$ has the following expression:
\begin{equation}\label{eqn:rank}
\left(
  \begin{array}{c}
    d\Phi_r(E_1) \\
    \vdots \\
    d\Phi_r(E_{2m-4}) \\
    d\Phi_r(E_{2m-3}) \\
    d\Phi_r(E_{2m-2}) \\
    d\Phi_r(E_{2m-1}) \\
  \end{array}
\right)=(B_{ij})
\left(
  \begin{array}{c}
    E_1^r \\
    \vdots \\
    E^r_{2m-4} \\
    E^r_{2m-3} \\
    E^r_{2m-2} \\
    E^r_{2m-1} \\
  \end{array}
\right),
\end{equation}
where
\begin{equation}\label{Bij111}
(B_{ij})=
\left(
\begin{array}{cccccc}
\cosh(r)-\lambda_1\sinh(r) & & & & & \\
 & \ddots & & & & \\
 & & \cosh(r)-\lambda_{2m-4}\sinh(r) & & & \\
  & &  & 1 & & \\
   & &  & & 1  & \\
    & &  & & & \cosh(2r)-\frac{\alpha}{2}\sinh(2r) \\
\end{array}
\right).
\end{equation}
The determinant of $(B_{ij})$ is given by
\begin{equation}\label{eqn:det}
{\rm Det}(B_{ij})=\prod_{i=1}^{2m-4}\big(\cosh(r)-\lambda_i\sinh(r)\big)
\big(\cosh(2r)-\frac{\alpha}{2}\sinh(2r)\big).
\end{equation}

Furthermore, let $S_r$ be the shape operator of the parallel hypersurface $\Phi_r(M)$
with respect to the unit normal vector field $\frac{d\exp_{p}(rN_{p})}{dr}$, where $p\in M$.
Then the expression of $S_r$ at $\Phi_r(p)$ is given by
\begin{equation}\label{eqn:Arr}
\left(
  \begin{array}{c}
    S_rE_1^r \\
    \vdots \\
    S_r E_{2m-4}^r \\
    S_r E_{2m-3}^r \\
    S_r E_{2m-2}^r \\
    S_r E_{2m-1}^r \\
  \end{array}
\right)=
(C_{ij})
\left(
  \begin{array}{c}
    E_1^r \\
    \vdots \\
    E_{2m-4}^r \\
    E_{2m-3}^r \\
    E_{2m-2}^r \\
    E_{2m-1}^r \\
  \end{array}
\right),
\end{equation}
where
\begin{equation}\label{Bij222}
(C_{ij})=
\left(
\begin{array}{cccccc}
\frac{-\sinh(r)+\lambda_1\cosh(r)}{\cosh(r)-\lambda_1\sinh(r)} & & & & & \\
 & \ddots & & & & \\
 & & \frac{-\sinh(r)+\lambda_{2m-4}\cosh(r)}{\cosh(r)-\lambda_{2m-4}\sinh(r)} & & & \\
  & &  & 0 & & \\
   & &  & & 0 & \\
    & &  & & & \frac{-2\sinh(2r)+\alpha\cosh(2r)}{\cosh(2r)-\frac{\alpha}{2}\sinh(2r)} \\
\end{array}
\right).
\end{equation}
\end{proposition}
\begin{proof}
At $p\in M$, let $N_p$ be a unit normal vector of $M$. Since $M$ has $\mathfrak{A}$-isotropic unit normal vector field, the four vectors
$N_p$, $\xi_p$, $AN_p$, $A\xi_p$ are pairwise orthonormal and the normal Jacobi operator
is given by
$$
\bar{R}_NZ=\bar{R}(Z,N)N=-Z+g(Z,N)N-3g(Z,JN)JN+g(Z,AN)AN+g(Z,JAN)JAN.
$$
It implies that $\bar{R}_N$ has three eigenvalues $ -1, -4, 0$ with corresponding
eigenspaces $\mathcal{Q}$, ${\rm Span}\{\xi_p\}$ and ${\rm Span}\{N_p,AN_p,A\xi_p\}$.

In the following, we use the Jacobi field method as described in [\cite{B-C-O}, Sec. 8.2]
to calculate the principal curvatures of the parallel hypersurface $\Phi_r(M)$ around
$M$. Let $\gamma$ be the
geodesic in $Q^{m*}$ with $\gamma(0)=p\in M$ and $\dot{\gamma}(0)=N_p$ and denote by $\gamma^\perp$ the parallel
subbundle of $TQ^{m*}$ along $\gamma$ defined by $\gamma^\perp_{\gamma(t)}
=T_{\gamma(t)}Q^{m*}\ominus \mathbb{R}\dot{\gamma}(t)$. Moreover, define
the $\gamma^\perp$-valued tensor field $\bar{R}^\perp_{\gamma}$ along $\gamma$ by
$\bar{R}^\perp_{\gamma(t)}X=\bar{R}(X,\dot{\gamma}(t))\dot{\gamma}(t)$. Now consider
the End$(\gamma^\perp)$-valued differential equation
$$
Y''+\bar{R}^\perp_\gamma\circ Y=0.
$$
Let $D$ be the unique solution of this differential equation with initial values
$$
D(0)=I,\ \ D'(0)=-S,
$$
where $I$ denotes the identity transformation. We decompose $\gamma^\perp_p$ further into
$$
\gamma^\perp_p=\oplus_{i=1}^{2m-4} V_{\lambda_i}\oplus{\rm Span}\{AN_p,A\xi_p\}\oplus{\rm Span}\{\xi_p\}.
$$
By explicit computation, we obtain \eqref{eqn:rank}. Moreover, the
shape operator $S_r$ of the parallel hypersurface $\Phi_r(M)$ around $M$ with respect to
$\dot{\gamma}(r)$ is given by $S_r=-D'(r)\circ D^{-1}(r)$.
Then by further computation, we can have \eqref{eqn:Arr}.
\end{proof}

From Proposition \ref{prop:4.1w}, the parallel hypersurface of a Hopf hypersurface $M$ of $Q^{m*}$ with constant principal curvatures and $\mathfrak{A}$-isotropic unit normal vector field has constant mean curvature
\begin{equation}\label{eqn:Hr}
H(r)={\rm Tr}S_r=\sum_{i=1}^{2m-4}\Big(\frac{-\sinh(r)+\lambda_i\cosh(r)}{\cosh(r)-\lambda_i\sinh(r)}\Big)
+\frac{-2\sinh(2r)+\alpha\cosh(2r)}{\cosh(2r)-\frac{\alpha}{2}\sinh(2r)}.
\end{equation}

Recall that let $M$ be an orientable hypersurface in a Riemannian manifold $(\overline{M}, g)$.
We say that $M$ is an isoparametric hypersurface of $\overline{M}$ if 
there exists an isoparametric function $F:\overline{M} \rightarrow \mathbb{R}$ such that $M=F^{-1}(l)$ 
for some regular value $l$ of $F$. Here, $F$ is called an isoparametric function
if the gradient and the Laplacian of $F$ satisfy
$$
\|\nabla F\|^{2}=f_1(F), \quad \Delta F=f_2(F),
$$
where $f_1, f_2: \mathbb{R} \rightarrow \mathbb{R}$ are smooth functions. 
There is an equivalent characterization of isoparametric hypersurfaces: a hypersurface in a Riemannian manifold is isoparametric if and only if its nearby parallel hypersurfaces have constant mean curvature.
It follows from \eqref{eqn:Arr} and \eqref{eqn:Hr} that a Hopf hypersurface $M$ of $Q^{m*}$ with constant principal curvatures and $\mathfrak{A}$-isotropic unit normal vector field is an isoparametric hypersurface. Moreover, all its parallel hypersurfaces also have constant principal curvatures. 
Note that any Hopf hypersurface of $Q^{m*}$ with constant principal curvatures has either $\mathfrak{A}$-principal unit normal vector field or $\mathfrak{A}$-isotropic unit normal vector field. Those with  $\mathfrak{A}$-principal unit normal vector field are open parts of the homogeneous examples described in Example \ref{ex:3.1}, Example \ref{ex:3.2} and Example \ref{ex:3.3}. Since homogeneous real hypersurfaces are naturally isoparametric hypersurfaces, we obtain the following result:
\begin{theorem}\label{thm:4.2w}
Let $M$ be a Hopf hypersurface of $Q^{m*}$ ($m\geq3$) with constant principal curvatures. Then $M$ is an isoparametric hypersurface. Moreover, all parallel hypersurfaces of $M$ have constant principal curvatures.
\end{theorem}

In the following, we study the focal submanifold (if it exists) of a Hopf hypersurface of $Q^{m*}$ with constant principal curvatures and $\mathfrak{A}$-isotropic unit normal vector field.
If $\Phi_r(M)$ is a focal
submanifold of $M$, then it holds ${\rm Det}(B_{ij})=0$ on $\Phi_r(M)$. In order to find the distance $|r|$, we observe the factors $\cosh(r)-\lambda_i\sinh(r)$ for $1\leq i\leq2m-4$ and $\cosh(2r)-\frac{\alpha}{2}\sinh(2r)$,
and find out the closest distance $|r|$ from $0$ (if it exists) such that one of these factors equals to $0$.
Recall that $\lambda_i\in \sigma(\mathcal{Q})$, $1\leq i\leq2m-4$.
Now, if $\alpha>2$ we define
$\lambda_+={\rm max}\{\lambda_i,\frac{\alpha+\sqrt{\alpha^2-4}}{2}|\lambda_i>0\}$ and $\lambda_-={\rm min}\{\lambda_i,-1|\lambda_i<0\}$; If $0\leq\alpha\leq2$, we define
$\lambda_+={\rm max}\{\lambda_i,1|\lambda_i>0\}$ and $\lambda_-={\rm min}\{\lambda_i,-1|\lambda_i<0\}$.

If $\lambda_+>1$ (or $\lambda_-<-1$), then $M$ has a focal submanifold. In that case,
the closest distance $r_1$ (or $|r_2|$) satisfies $\coth(r_1)=\lambda_+$ and $r_1>0$ (or $\coth(r_2)=\lambda_-$ and $r_2<0$). 

For the distance $r_1$ (if it exists), according to the computation of Proposition \ref{prop:4.1w}, $M_+:=\Phi_{r_1}(M)$ is a focal
submanifold of $M$. If $0\leq\alpha\leq2$ and $\lambda_+>1$, then
${\rm dim}M_+=2m-{\rm dim}V_{\lambda_+}-1$;
If $\alpha>2$ and $\lambda_+>\frac{\alpha+\sqrt{\alpha^2-4}}{2}$, then
${\rm dim}M_+=2m-{\rm dim}V_{\lambda_+}-1$; If $\alpha>2$ and $\lambda_+=\frac{\alpha+\sqrt{\alpha^2-4}}{2}$,
then ${\rm dim}M_+=2m-{\rm dim}V_{\lambda_+}-2$.
Here, $V_{\lambda_+}$ is the corresponding eigenspace
restricted to $\mathcal{Q}$.
We recall that a submanifold is said to be austere if its multiset of principal curvatures is invariant under change of sign. In particular, austere submanifolds are automatically minimal.
This concept was introduced by Harvey-Lawson \cite{HL} for constructing special
Lagrangian submanifolds in $\mathbb{C}^n$.
Now, we have the following result about $M_+$.

\begin{proposition}\label{prop:4.3w}
Let $M$ be a Hopf hypersurface of $Q^{m*}$ ($m\geq3$) with constant principal curvatures and $\mathfrak{A}$-isotropic unit normal vector field $N$. If $\lambda_+>1$, then the focal submanifold $M_+$ has constant principal curvatures with respect to any unit normal vector field, and $M_+$ is austere.
Moreover, if $\alpha>2$ and $\lambda_+>\frac{\alpha+\sqrt{\alpha^2-4}}{2}$, then
the constant principal curvatures of $M_+$ are
$$
0,\ \ \ \alpha+\frac{(\alpha^2-4)\lambda_+}{\lambda_+^2-\alpha\lambda_++1},
\ \ \ \ \frac{\lambda_+\lambda_i-1}{\lambda_+-\lambda_i},\ \  {\rm where}\ \  \lambda_i<\lambda_+.
$$
If $\alpha>2$ and $\lambda_+=\frac{\alpha+\sqrt{\alpha^2-4}}{2}$, then
the constant principal curvatures of $M_+$ are
$$
0,\ \ \ \frac{\lambda_+\lambda_i-1}{\lambda_+-\lambda_i},\ \  {\rm where}\ \  \lambda_i<\lambda_+.
$$
If $0\leq\alpha\leq2$ and $\lambda_+>1$, then
the constant principal curvatures of $M_+$ are
$$
0,\ \ \ \alpha+\frac{(\alpha^2-4)\lambda_+}{\lambda_+^2-\alpha\lambda_++1},
\ \ \ \ \frac{\lambda_+\lambda_i-1}{\lambda_+-\lambda_i},\ \  {\rm where}\ \  \lambda_i<\lambda_+.
$$
For each of the above situations, it holds that
$$
\frac{\lambda_+\lambda_i-1}{\lambda_+-\lambda_i}
<\frac{\lambda_+\lambda_j-1}{\lambda_+-\lambda_j},
\ \ {\rm for}\ \ \lambda_i<\lambda_j<\lambda_+.
$$
\end{proposition}
\begin{proof}
According to Proposition \ref{prop:4.1w} and Theorem \ref{thm:4.2w},
$M$ is an isoparametric hypersurface,
and all parallel hypersurfaces of $M$ have constant principal curvatures.
Then, by Theorem 1.2 obtained by Ge-Tang \cite{GT}, the focal submanifold $M_+$ is austere.
Furthermore, by the Jacobi field method as used in Proposition \ref{prop:4.1w} and  taking $r=r_1$,we can get that the constant principal curvatures of $M_+$ with respect to any unit normal vector field are given as above.

If we fix $\lambda_+$ , then the function
$\frac{{\lambda_+}\lambda_i-1}{{\lambda_+}-\lambda_i}$ is increasing with respect to $\lambda_i$
for $\lambda_i<\lambda_+$.
\end{proof}

For the distance $|r_2|$, $r_2<0$ (if it exists), according to the computation of Proposition \ref{prop:4.1w}, $M_-:=\Phi_{r_2}(M)$ is a focal
submanifold of $M$. If $\lambda_-<-1$, then
${\rm dim}M_-=2m-{\rm dim}V_{\lambda_-}-1$.
Here, $V_{\lambda_-}$ is the corresponding eigenspace
restricted to $\mathcal{Q}$.
Then, similar as the proof of Proposition \ref{prop:4.3w}, we have the following result about $M_-$.

\begin{proposition}\label{prop:4.4w}
Let $M$ be a Hopf hypersurface of $Q^{m*}$ ($m\geq3$) with constant principal curvatures and $\mathfrak{A}$-isotropic unit normal vector field $N$. If $\lambda_-<-1$, then the focal submanifold $M_-$ has constant principal curvatures with respect to any unit normal vector field, and $M_-$ is austere.
Moreover, the constant principal curvatures of $M_-$ are
$$
0,\ \ \ \alpha+\frac{(\alpha^2-4){\lambda_-}}{\lambda_-^2-\alpha{\lambda_-}+1}, \ \ \ \ \frac{{\lambda_-}\lambda_i-1}{{\lambda_-}-\lambda_i}\ \  {\rm where}\ \  {\lambda_-}<\lambda_i.
$$
For this situation, it holds that
$$
\frac{{\lambda_-}\lambda_i-1}{{\lambda_-}-\lambda_i}
<\frac{{\lambda_-}\lambda_j-1}{{\lambda_-}-\lambda_j},
\ \ {\rm for}\ \ {\lambda_-}<\lambda_i<\lambda_j.
$$
\end{proposition}

\begin{remark}\label{rem:4.1aa}
If $\lambda_+>\frac{\alpha+\sqrt{\alpha^2-4}}{2}$ (or $\lambda_-<-1$), then to order the principal curvatures of $M_+$ (or $M_-$) in increasing magnitude, it suffices to analyze the relationship between $\alpha+\frac{(\alpha^2-4){\lambda_+}}{{\lambda_+}^2-\alpha{\lambda_+}+1}$ and $\frac{{\lambda_+}\lambda_i-1}{{\lambda_+}-\lambda_i}$ (or between $\alpha+\frac{(\alpha^2-4){\lambda_-}}{{\lambda_-}^2-\alpha{\lambda_-}+1}$ and $\frac{{\lambda_-}\lambda_i-1}{{\lambda_-}-\lambda_i}$). Subsequently, leveraging the fact that $M_+$ (or $M_-$) is austere, one can derive the specific relations among the principal curvatures $\{\alpha,0,\lambda_1,...,\lambda_{2m-4}\}$ of the original hypersurface $M$. 
\end{remark}

\subsection{Cartan's formulas for Hopf hypersurfaces with constant principal curvatures and $\mathfrak{A}$-isotropic unit normal vector field}\label{sect:4.2}~

Let $M$ be a Hopf hypersurface with constant principal curvatures and $\mathfrak{A}$-isotropic unit normal vector field $N$.
Suppose tangent vector fields $X,Y\in TM$ satisfy $SX=\lambda X$ and $SY=\mu Y$,
then for any tangent vector $Z$, it holds
$$
g((\nabla_Z S)X,Y)=(\lambda-\mu)g(\nabla_Z X,Y).
$$

Since $\mathcal{Q}$ is $S$-invariant, we denote the set
of eigenvalues of $S$ restricted to $\mathcal{Q}$ by $\sigma(\mathcal{Q})$.
For any $\lambda\in\sigma(\mathcal{Q})$, let $V_\lambda$ be the corresponding eigenspace
restricted to $\mathcal{Q}$.

\begin{lemma}\label{lemma:5.1}
Let $M$ be a Hopf hypersurface of $Q^{m*}$ ($m\geq3$) with constant principal curvatures and $\mathfrak{A}$-isotropic unit normal vector field $N$. For any $\lambda\in\sigma(\mathcal{Q})$ and any almost product structure $A\in\mathfrak{A}$,
we have
\begin{enumerate}
\item[(1)]
$\nabla_X (AN)=q(X)A\xi-\lambda AX$ for all $X\in V_\lambda$, where $q$ is the corresponding $1$-form of $A$ described in \eqref{eqn:2.2};

\item[(2)]
$\nabla_X (A\xi)=-q(X)AN-\lambda JAX$ for all $X\in V_\lambda$, where $q$ is the corresponding $1$-form of $A$
described in \eqref{eqn:2.2};

\item[(3)]
For $all\  X,Y\in V_\lambda$, we have $(\lambda^2-1)g(AX,Y)=(\lambda^2-1)g(AX,JY)=0$. Furthermore,
if $\lambda\neq\pm1$, then $AV_\lambda\bot {\rm Span}\{V_\lambda, JV_\lambda\}$, i.e.
$$
g(AX,Y)=g(AX,JY)=0,\ for\ all\  X,Y\in V_\lambda.
$$
\end{enumerate}
\end{lemma}
\begin{proof}
For (1), by \eqref{eqn:2.2}, we have
$$
\begin{aligned}
\nabla_X (AN)&=\bar{\nabla}_X (AN)-g(X,SAN)N=\bar{\nabla}_X (AN)=(\bar{\nabla}_X A)N+A\bar{\nabla}_X N\\
&=q(X)JAN-ASX=q(X)A\xi-\lambda AX, \ \ \forall\ X\in V_\lambda.
\end{aligned}
$$

For (2), by \eqref{eqn:2.2}, we have
$$
\begin{aligned}
\nabla_X (A\xi)&=\bar{\nabla}_X (A\xi)-g(X,SA\xi)N=\bar{\nabla}_X (A\xi)=(\bar{\nabla}_X A)\xi+A\bar{\nabla}_X \xi\\
&=q(X)JA\xi+A\nabla_X \xi=-q(X)AN+A\phi SX=-q(X)AN+\lambda A\phi X\\
&=-q(X)AN-\lambda JAX, \ \ \forall\ X\in V_\lambda.
\end{aligned}
$$

For (3), on the one hand, by Codazzi equation \eqref{eqn:2.9}, we have
$$
g((\nabla_X S)AN-(\nabla_{AN} S)X,Y)=g(AX,Y), \ \ \forall\ X,Y\in V_\lambda.
$$
On the other hand, by above (1), we can have
$$
\begin{aligned}
g((\nabla_X S)AN&-(\nabla_{AN} S)X,Y)=g(-S\nabla_X (AN),Y)=-\lambda g(\nabla_X (AN),Y)\\
&=-\lambda g(q(X)JAN-\lambda AX,Y)=\lambda^2 g(AX,Y), \ \ \forall\ X,Y\in V_\lambda.
\end{aligned}
$$
Combining above two equations, we have $(\lambda^2-1)g(AX,Y)=0$, for all $X,Y\in V_\lambda$.

Similarly, by Codazzi equation \eqref{eqn:2.9}, we have
$$
g((\nabla_X S)A\xi-(\nabla_{A\xi} S)X,Y)=g(JAX,Y)=-g(AX,JY), \ \ \forall\ X,Y\in V_\lambda.
$$
On the other hand, by above (2), we can have
$$
\begin{aligned}
g((\nabla_X S)A\xi&-(\nabla_{A\xi} S)X,Y)=g(-S\nabla_X (A\xi),Y)=-\lambda g(\nabla_X (A\xi),Y)\\
&=-\lambda g(-q(X)AN-\lambda JAX,Y)=-\lambda^2 g(AX,JY), \ \ \forall\ X,Y\in V_\lambda.
\end{aligned}
$$
Combining above two equations, we have $(\lambda^2-1)g(AX,JY)=0$, for all $X,Y\in V_\lambda$.
\end{proof}

\begin{remark}\label{rem:4.1sss}
In a recent paper \cite{LTY}, the authors established a useful lemma concerning Hopf hypersurfaces in 
$Q^m$ with constant principal curvatures and $\mathfrak{A}$-isotropic unit normal vector field. 
The lemma states that the almost product structure $A$ maps each eigenspace 
$V_{\lambda}\subset \mathcal{Q}$ onto the orthogonal complement of $V_{\lambda}\oplus JV_{\lambda}$ 
(see Lemma 4.5 (3) in \cite{LTY}). However, the situation in $Q^{m*}$ is quite different. In particular, for 
$\lambda=\pm1$, the image $AV_\lambda$ and the subspace ${\rm Span}\{V_\lambda, JV_\lambda\}$  
are not necessarily orthogonal. Indeed, as shown in \cite{B-2023}, on each of Examples \ref{ex:3.6}--\ref{ex:3.10}, there exists an almost product structure $A\in\mathfrak{A}$ 
such that $AV_1=V_1$ and $AV_{-1}=V_{-1}$. Due to this distinction, the corresponding Cartan's formulas (see \eqref{eqn:ca2}, \eqref{eqn:call} and \eqref{eqn:call2g} below) in $Q^{m*}$ differ notably from those in 
$Q^m$.
\end{remark}

\begin{lemma}\label{lemma:5.2}
Let $M$ be a Hopf hypersurface of $Q^{m*}$ ($m\geq3$) with constant principal curvatures and $\mathfrak{A}$-isotropic unit normal vector field $N$. For all $\lambda,\mu$ in $\sigma(\mathcal{Q})$, we have
\begin{enumerate}
\item[(1)]
$\nabla_X Y+\lambda g(\phi X,Y)\xi-\lambda g(AX,Y)AN+\lambda g(AX,JY)A\xi\in V_\lambda$ for all $X,Y\in V_\lambda$;

\item[(2)]
$\nabla_X Y\bot V_\lambda$ if $X\in V_\lambda$, $Y\in V_\mu$, $\lambda\neq\mu$.
\end{enumerate}
\end{lemma}
\begin{proof}
By Lemma \ref{lemma:5.1}, for any $X,Y\in V_\lambda$, we have
$$
g(\nabla_X Y,\xi)=-g(Y,\nabla_X \xi)=-g(Y,\phi SX)=-\lambda g(\phi X,Y),
$$
$$
g(\nabla_X Y,AN)=-g(Y,\nabla_X (AN))=-g(Y,q(X)A\xi-\lambda AX)=\lambda g(AX,Y),
$$
$$
g(\nabla_X Y,A\xi)=-g(Y,\nabla_X (A\xi))=-g(Y,-q(X)AN-\lambda JAX)=-\lambda g(AX,JY).
$$
It is easy to see that $\nabla_X Y+\lambda g(\phi X,Y)\xi-\lambda g(AX,Y)AN+\lambda g(AX,JY)A\xi\in\mathcal{Q}$.
Now take any $\mu\in \sigma(\mathcal{Q})$ with $\mu\neq\lambda$ and choose any $Z\in V_\mu$.
By the Codazzi equation,
$$
\begin{aligned}
0=g((\nabla_X S)Z&-(\nabla_Z S)X,Y)=g(\mu\nabla_X Z-S\nabla_X Z,Y)=(\mu-\lambda)g(\nabla_X Z,Y)\\
&=(\lambda-\mu)g(\nabla_X Y,Z), \ \ \forall\ X,Y\in V_\lambda, \ Z\in V_\mu, \ \mu\neq\lambda.
\end{aligned}
$$
Thus, $\nabla_X Y+\lambda g(\phi X,Y)\xi-\lambda g(AX,Y)AN+\lambda g(AX,JY)A\xi\in V_\lambda$. Note that
the second assertion follows from the first since $g(\nabla_X Z,Y)=-g(\nabla_X Y,Z)$.
\end{proof}



In the following, we give Cartan's formulas (see \eqref{eqn:ca2}, \eqref{eqn:call}
and \eqref{eqn:call2g}) for the Hopf hypersurface of $Q^{m*}$ ($m\geq3$) with constant principal curvatures and $\mathfrak{A}$-isotropic unit normal vector field $N$, which play an important role in the subsequent proofs.

\begin{lemma}\label{lemma:5.3}
Let $M$ be a Hopf hypersurface of $Q^{m*}$ ($m\geq3$) with constant principal curvatures and $\mathfrak{A}$-isotropic unit normal vector field $N$. Let $X\in \mathcal{Q}$ be a unit principal vector at a point $p$ with
associated principal curvature $\lambda$. For any principal orthonormal basis $\{e_i\}_{i=1}^{2m-4}$ of
$\mathcal{Q}$ satisfying $Se_i=\mu_ie_i$, we have
\begin{equation}\label{eqn:ca2}
\begin{aligned}
\sum_{i=1, \mu_i\neq\lambda}^{2m-4}\frac{\lambda\mu_i-1}{\lambda-\mu_i}\Big\{&1+2g(\phi X,e_i)^2-2g(AX,e_i)^2-2g(AX,Je_i)^2\\
&+g(AX,X)g(Ae_i,e_i)+g(AX,JX)g(Ae_i,Je_i)\Big\}=0.
\end{aligned}
\end{equation}
\end{lemma}
\begin{proof}
Let $\mu\neq\lambda$, $\mu\in\sigma(\mathcal{Q})$ and $Y\in \mathcal{Q}$ be another unit
principal vector at $p$ with corresponding principal
curvature $\mu$. Extend $X$ and $Y$ to be principal vector fields near $p$.

(1). First note that $g(\nabla_X Y,\xi)=-g(Y,\nabla_X \xi)=-g(Y,\phi SX)=-\lambda g(\phi X,Y)$. Similarly,
$$
g(\nabla_Y X,\xi)=-\mu g(\phi Y,X)=\mu g(\phi X,Y).
$$
Thus, $g([X,Y],\xi)=-(\lambda+\mu)g(\phi X,Y)$.

By $g(\nabla_X Y,AN)=-g(Y,\nabla_X (AN))=-g(Y,q(X)A\xi-\lambda AX)=\lambda g(A X,Y)$, and
$g(\nabla_Y X,AN)=\mu g(AX,Y)$, we have $g([X,Y],AN)=(\lambda-\mu)g(AX,Y)$.

By $g(\nabla_X Y,A\xi)=-g(Y,\nabla_X (A\xi))=g(Y,q(X)AN+\lambda JAX)=-\lambda g(AX,JY)$ and
$g(\nabla_Y X,A\xi)=-\mu g(AX,JY)$, we have $g([X,Y],A\xi)=-(\lambda-\mu)g(AX,JY)$.

Next, using Codazzi equation \eqref{eqn:2.9}, we compute
\begin{equation}\label{eqn:5.9}
\begin{aligned}
g((\nabla_{[X,Y]}S)X,Y)=&g((\nabla_X S)[X,Y],Y)-g([X,Y],\xi)g(\phi X,Y)\\
&-g([X,Y],AN)g(AX,Y)-g([X,Y],A\xi)g(JAX,Y)\\
=&g((\nabla_X S)[X,Y],Y)+(\lambda+\mu)g(\phi X,Y)^2\\
&-(\lambda-\mu)g(AX,Y)^2-(\lambda-\mu)g(AX,JY)^2\\
=&g((\nabla_Y S)X,[X,Y])+2g(\phi X,Y)g([X,Y],\xi)\\
&+(\lambda+\mu)g(\phi X,Y)^2-(\lambda-\mu)\{g(AX,Y)^2-g(AX,JY)^2\}\\
=&g((\nabla_Y S)X,[X,Y])-(\lambda+\mu)g(\phi X,Y)^2\\
&-(\lambda-\mu)g(AX,Y)^2-(\lambda-\mu)g(AX,JY)^2.
\end{aligned}
\end{equation}
But now,
$$
g(\nabla_X Y,(\nabla_Y S)X)=g((\lambda I-S)\nabla_Y X,\nabla_X Y),
$$
while
$$
\begin{aligned}
g(\nabla_Y X,(\nabla_Y S)X)&=g(\nabla_Y X,(\nabla_X S)Y)+2g(\phi Y,X)g(\nabla_Y X,\xi)\\
&=g((\nabla_X S)Y,\nabla_Y X)-2g(\phi X,Y)g(\nabla_Y X,\xi)\\
&=g((\mu I-S)\nabla_X Y,\nabla_Y X)-2\mu g(\phi X,Y)^2.
\end{aligned}
$$
Thus
\begin{equation}\label{eqn:5.10}
g([X,Y],(\nabla_Y S)X)=(\lambda-\mu)g(\nabla_X Y,\nabla_Y X)+2\mu g(\phi X,Y)^2.
\end{equation}
On substituting in equation \eqref{eqn:5.9} we obtain
\begin{equation}\label{eqn:5.1}
g((\nabla_{[X,Y]} S)X,Y)=(\lambda-\mu)\Big\{g(\nabla_X Y,\nabla_Y X)
-g(\phi X,Y)^2-g(AX,Y)^2-g(AX,JY)^2\Big\}.
\end{equation}

(2). By using the Gauss equation \eqref{eqn:2.8}, we have
\begin{equation}\label{eqn:5.2}
\begin{aligned}
g(R(X,Y)Y,X)=&-1+\lambda\mu-3g(\phi X,Y)^2+g(AX,Y)^2+g(AX,JY)^2\\
&-g(AX,X)g(AY,Y)-g(AX,JX)g(AY,JY).
\end{aligned}
\end{equation}

(3). Note that $g(\nabla_Y Y,X)=0$ by Lemma \ref{lemma:5.2}, and so
$$
\begin{aligned}
g(\nabla_X {\nabla_Y Y},X)&=-g(\nabla_Y Y,\nabla_X X)=-g(\nabla_Y Y,AN)g(\nabla_X X,AN)-g(\nabla_Y Y,A\xi)g(\nabla_X X,A\xi)\\
&=-\lambda\mu [g(AX,X)g(AY,Y)+g(AX,JX)g(AY,JY)].
\end{aligned}
$$
Similarly, $g(\nabla_X Y,X)=0$, so that
$$
g(\nabla_Y {\nabla_X Y},X)=-g(\nabla_X Y,\nabla_Y X).
$$
So, $g((\nabla_{[X,Y]} S)X,Y)=(\lambda-\mu)g(\nabla_{[X,Y]} X,Y)$. We then compute
$$
\begin{aligned}
g(\nabla_X{\nabla_Y Y}-\nabla_Y {\nabla_X Y}-\nabla_{[X,Y]} Y,X)=&-\lambda\mu [g(AX,X)g(AY,Y)+g(AX,JX)g(AY,JY)]\\
&+g(\nabla_X Y,\nabla_Y X)
+\frac{1}{\lambda-\mu}g((\nabla_{[X,Y]} S)X,Y),
\end{aligned}
$$
which gives
\begin{equation}\label{eqn:5.3}
\begin{aligned}
g(R(X,Y)Y,X)=&-\lambda\mu[g(AX,X)g(AY,Y)+g(AX,JX)g(AY,JY)]\\
&+g(\nabla_X Y,\nabla_Y X)+\frac{1}{\lambda-\mu}g((\nabla_{[X,Y]} S)X,Y).
\end{aligned}
\end{equation}

(4). For a unit principal vector $Z\in \mathcal{Q}$
corresponding to a principal curvature $\nu$ not equal to $\lambda$ or $\mu$. Extend $Z$ to be principal vector fields near $p$.
By using the Codazzi equation \eqref{eqn:2.9}, we compute
$$
g((\nabla_Z S)X,Y)=g((\nabla_X S)Z,Y)=g(Z,(\nabla_X S)Y)=(\mu-\nu)g(Z,\nabla_X Y).
$$
The same calculation with $X$ and $Y$ interchanged gives
$$
g((\nabla_Z S)X,Y)=(\lambda-\nu)g(Z,\nabla_Y X).
$$
Multiplying these two equations together gives
\begin{equation}\label{eqn:5.4}
(\lambda-\nu)(\mu-\nu)g(\nabla_X Y,Z)g(\nabla_Y X,Z)=g((\nabla_Z S)X,Y)^2.
\end{equation}

(5). Note that
$$
g(\nabla_X Y,\xi)g(\nabla_Y X,\xi)=g(\phi SX,Y)g(\phi SY,X)=\lambda\mu g(\phi X,Y)g(\phi Y,X)=-\lambda\mu g(\phi X,Y)^2,
$$
$$
\begin{aligned}
g(\nabla_X Y,AN)&g(\nabla_Y X,AN)=g(Y,\nabla_X (AN))g(X,\nabla_Y (AN))\\
&=g(Y,q(X)A\xi-\lambda AX)g(X,q(Y)A\xi-\mu AY)
=\lambda\mu g(AX,Y)^2,
\end{aligned}
$$
$$
\begin{aligned}
g(\nabla_X Y,A\xi)&g(\nabla_Y X,A\xi)=g(Y,\nabla_X (A\xi))g(X,\nabla_Y (A\xi))\\
&=g(Y,-q(X)AN-\lambda JAX)g(X,-q(Y)AN-\mu JAY)=\lambda\mu g(AX,JY)^2.
\end{aligned}
$$
Thus, in order to express $g(\nabla_X Y,\nabla_Y X)$ in terms of the orthonormal principal basis,
we need only observe that the terms omitted from the full summation of
the $g(\nabla_X Y,e_i)g(\nabla_Y X,e_i)$ (i.e., those $i$ for which $\mu_i=\lambda$ or $\mu_i=\mu$
actually vanish and make no contribution to the sum). This is easily checked using Lemma \ref{lemma:5.2}.
Then, we can have
\begin{equation}\label{eqn:5.5}
g(\nabla_X Y,\nabla_Y X)=\sum_{\mu_i\neq\lambda,\mu}g(\nabla_X Y,e_i)g(\nabla_Y X,e_i)
-\lambda\mu\Big\{g(\phi X,Y)^2-g(AX,Y)^2-g(AX,JY)^2\Big\}.
\end{equation}

(6). Combining equations \eqref{eqn:5.1}, \eqref{eqn:5.2} and \eqref{eqn:5.3}, we get
$$
\begin{aligned}
2g(\nabla_X Y,\nabla_Y X)=&-1+\lambda\mu-2g(\phi X,Y)^2+2g(AX,Y)^2+2g(AX,JY)^2\\
&+(\lambda\mu-1)[g(AX,X)g(AY,Y)+g(AX,JX)g(AY,JY)].
\end{aligned}
$$
Then, we use this and the result of equation \eqref{eqn:5.4} in equation \eqref{eqn:5.5}, we get
\begin{equation}\label{eqn:5.6}
\begin{aligned}
2\sum_{\mu_i\neq\lambda,\mu}\frac{g((\nabla_{e_i} S)X,Y)^2}{(\lambda-\mu_i)(\mu-\mu_i)}
=&(\lambda\mu-1)\Big\{1+2g(\phi X,Y)^2-2g(AX,Y)^2-2g(AX,JY)^2\\
&   +g(AX,X)g(AY,Y)+g(AX,JX)g(AY,JY)\Big\}.
\end{aligned}
\end{equation}

Now for any $j$ with $\mu_j\neq\lambda$, we have (setting $Y=e_j$ in equation \eqref{eqn:5.6}),
\begin{equation}\label{eqn:5.7}
\begin{aligned}
2\sum_{\mu_i\neq\lambda,\mu_j}\frac{g((\nabla_{e_i} S)e_j,X)^2}{(\lambda-\mu_i)(\lambda-\mu_j)(\mu_j-\mu_i)}
=&\frac{\lambda\mu_j-1}{\lambda-\mu_j}\Big\{1+2g(\phi X,e_j)^2-2g(AX,e_j)^2-2g(AX,Je_j)^2\\
&   +g(AX,X)g(Ae_j,e_j)+g(AX,JX)g(Ae_j,Je_j)\Big\}.
\end{aligned}
\end{equation}

Summing this over all $j$ for which $\mu_j\neq\lambda$, we have
\begin{equation}\label{eqn:5.8}
\begin{aligned}
2\sum_{i,j;\mu_i\neq\mu_j;\mu_i,\mu_j\neq\lambda}&\frac{g((\nabla_{e_i} S)e_j,X)^2}{(\lambda-\mu_i)(\lambda-\mu_j)(\mu_j-\mu_i)}\\
&=\sum_{\mu_j\neq\lambda}\frac{\lambda\mu_j-1}{\lambda-\mu_j}\Big\{1+2g(\phi X,e_j)^2-2g(AX,e_j)^2-2g(AX,Je_j)^2\\
&\qquad\qquad\qquad\qquad  +g(AX,X)g(Ae_j,e_j)+g(AX,JX)g(Ae_j,Je_j)\Big\}.
\end{aligned}
\end{equation}

Since the summand on the left side of equation \eqref{eqn:5.8} is skew-symmetric in $\{i,j\}$,
the value of the sum is $0$, and so the sum on the right is $0$.
\end{proof}


From equation \eqref{eqn:5.6}, we have the following lemma:

\begin{lemma}\label{lemma:5.4}
Let $M$ be a Hopf hypersurface of $Q^{m*}$ ($m\geq3$) with constant principal curvatures and $\mathfrak{A}$-isotropic unit normal vector field $N$. Let $X,Y\in \mathcal{Q}$ be unit principal vector fields with
associated principal curvatures $\lambda,\mu$ ($\mu\neq\lambda$). For any principal orthonormal basis $\{e_i\}_{i=1}^{2m-4}$ of
$\mathcal{Q}$ satisfying $Se_i=\mu_ie_i$, we have
\begin{equation}\label{eqn:call}
\begin{aligned}
2\sum_{\mu_i\neq\lambda,\mu}\frac{g((\nabla_{e_i} S)X,Y)^2}{(\lambda-\mu_i)(\mu-\mu_i)}
=&(\lambda\mu -1)\Big\{1+2g(\phi X,Y)^2-2g(AX,Y)^2-2g(AX,JY)^2\\
&  +g(AX,X)g(AY,Y)+g(AX,JX)g(AY,JY)\Big\}.
\end{aligned}
\end{equation}
\end{lemma}

If the number of distinct principal curvatures of $\sigma(\mathcal{Q})$ is $2$, then we have the following result.

\begin{lemma}\label{lemma:4.1aa}
Let $M$ be a Hopf hypersurface of $Q^{m*}$ ($m\geq3$) with constant principal curvatures and $\mathfrak{A}$-isotropic unit normal vector field $N$. Assume that the number of distinct principal curvatures of $\sigma(\mathcal{Q})$ is $2$, i.e.
$\sigma(\mathcal{Q})=\{\lambda,\mu\}$, $\mu\neq\lambda$. Then for any $X\in \mathcal{Q}$ and $Y\in \mathcal{Q}$ being unit principal vector fields with
associated principal curvatures $\lambda$ and $\mu$ respectively, we have
\begin{equation}\label{eqn:call2g}
\begin{aligned}
0=(\lambda\mu -1)\Big\{&1+2g(\phi X,Y)^2-2g(AX,Y)^2-2g(AX,JY)^2\\
&+g(AX,X)g(AY,Y)+g(AX,JX)g(AY,JY)\Big\}.
\end{aligned}
\end{equation}
\end{lemma}
\begin{proof}
According to the assumption that the number of distinct principal curvatures of $\sigma(\mathcal{Q})$ is $2$, then following the
proof of Lemma \ref{lemma:5.3}, we have
$$
g(\nabla_X Y,\nabla_Y X)=-\lambda\mu\Big\{g(\phi X,Y)^2-g(AX,Y)^2-g(AX,JY)^2\Big\},
$$
and
$$
\begin{aligned}
2g(\nabla_X Y,\nabla_Y X)=&-1+\lambda\mu-2g(\phi X,Y)^2+2g(AX,Y)^2+2g(AX,JY)^2\\
&+(\lambda\mu-1)[g(AX,X)g(AY,Y)+g(AX,JX)g(AY,JY)].
\end{aligned}
$$
Combining these two equations, we obtain \eqref{eqn:call2g}.
\end{proof}

Finally, as a direct application of Lemma \ref{lemma:2.3},
we have the following useful lemma.

\begin{lemma}\label{lemma:5.5}
Let $M$ be a Hopf hypersurface of $Q^{m*}$ ($m\geq3$) with constant principal curvatures
and $\mathfrak{A}$-isotropic unit normal vector field $N$, and it holds $S\xi=\alpha\xi$ ($\alpha\geq0$) on $M$.
Then for every $\lambda\in\sigma(\mathcal{Q})$ and $X\in V_\lambda$, it holds
\begin{equation}\label{eqn:SS1}
(2\lambda-\alpha)S\phi X=(\alpha\lambda-2)\phi X.
\end{equation}
Moreover, we have
\begin{enumerate}
\item[(1)]
When $\alpha\neq2$ on $M$, then for every $\lambda\in\sigma(\mathcal{Q})$, there is a unique $\mu\in\sigma(\mathcal{Q})$ such that
$\mu=\tfrac{\alpha\lambda-2}{2\lambda-\alpha}$ and $\phi V_{\lambda}=V_{\mu}$.

\item[(2)]
When $\alpha=2$ on $M$, then $1\in\sigma(\mathcal{Q})$. For every $\lambda\in\sigma(\mathcal{Q})$
and $\lambda\neq1$, it holds
$\phi V_{\lambda}\subset V_{1}$.
Let $\{\lambda_i\}_{i=1}^{k}$ be all the distinct principal curvatures in $\sigma(\mathcal{Q})$,
then $V_1\ominus\phi(\oplus_{\lambda_i\neq1}V_{\lambda_i})$,
the orthogonal complement of
$\phi(\oplus_{\lambda_i\neq1}V_{\lambda_i})$ in $V_1$, is $J$-invariant.
\end{enumerate}
\end{lemma}
\begin{proof}
By $M$ has $\mathfrak{A}$-isotropic unit normal vector field $N$, then $\mathcal{Q}$ is $S$-invariant and $J$-invariant.
Now, taking a unit principal vector field $X\in\mathcal{Q}$ such that $SX=\lambda X$, then by \eqref{eqn:2.13},
we obtain \eqref{eqn:SS1}. If $2\lambda-\alpha=0$, then it holds that
$\lambda=\frac{\alpha}{2}$ and $\alpha\lambda-2=0$, which deduces $\alpha=2$ and $\lambda=1$.

Now, we divide the discussion into two cases, depending on the value of $\alpha$.

(1) When $\alpha\neq2$ on $M$, then we have $S\phi X=\frac{\alpha\lambda-2}{2\lambda-\alpha}\phi X$ for any $X\in V_\lambda$. It means that $\mu=\tfrac{\alpha\lambda-2}{2\lambda-\alpha}$ is also a principal curvature in $\mathcal{Q}$, and it holds $\phi V_{\lambda}\subset V_{\mu}$.
Conversely, for any $X\in V_{\mu}$, then by \eqref{eqn:SS1}, we have $S\phi X=\lambda \phi X$.
So $\phi V_{\lambda}=V_{\mu}$.

(2) When $\alpha=2$ on $M$, for every $\lambda\in\sigma(\mathcal{Q})$
and $\lambda\neq1$ (if it exists), we have $(2\lambda-2)S\phi X=(2\lambda-2)\phi X$,
which means that $1\in\sigma(\mathcal{Q})$ and $\phi V_{\lambda}\subset V_{1}$.
According to $\oplus_{\lambda_i\neq1}(V_{\lambda_i}\oplus \phi V_{\lambda_i})$ and
$\mathcal{Q}=(\oplus_{\lambda_i\neq1}V_{\lambda_i})\oplus V_1$ are $J$-invariant, so
it follows that $V_1\ominus\phi(\oplus_{\lambda_i\neq1}V_{\lambda_i})$ is $J$-invariant.


Thus, we have proved (1) and (2) of Lemma \ref{lemma:5.5}.
\end{proof}

\section{Proofs of Theorems \ref{thm:1.1a}, \ref{thm:1.1b} and \ref{thm:1.1c}}\label{sect:5}

According to Lemma \ref{lemma:2.5}, any Hopf hypersurface in $Q^{m*}$ $(m\ge3)$ with constant
principal curvatures has either $\mathfrak{A}$-principal unit normal vector field or $\mathfrak{A}$-isotropic unit normal vector field. Moreover,
a Hopf hypersurface in $Q^{m*}$ $(m\ge3)$ with $\mathfrak{A}$-principal unit normal vector field is either an open part of Example \ref{ex:3.1} or Example \ref{ex:3.2} or Example \ref{ex:3.3}
(see Theorem \ref{thm:2.7}). Thus, we only need to consider the Hopf hypersurfaces with constant principal curvatures and $\mathfrak{A}$-isotropic unit normal vector field. Since the Reeb function $\alpha$ is  
constant, without loss of generality, up to a sign of the unit normal vector field $N$, we can always assume that $\alpha\geq0$.

\subsection{The proof of Theorem \ref{thm:1.1a}}\label{sect:5.1}~

In this subsection, we assume that $M$ is a Hopf hypersurface of $Q^{m*}$ ($m\geq3$) with at most two distinct constant principal curvatures. There are only the following two cases: 

{\bf Case-1}: $M$ has one distinct constant principal curvatures.

{\bf Case-2}: $M$ has two distinct constant principal curvatures.

In {\bf Case-1}, $M$ is a totally umbilical real hypersurface of $Q^{m*}$ ($m\geq3$), 
then $S=\lambda\,{\rm Id}$, so it holds $S\phi=\phi S$ on $M$, which means that $M$ has isometric Reeb flow. According the Theorem \ref{thm:2.2} and the principal curvatures of Example \ref{ex:3.4} and
Example \ref{ex:3.5}, we know that there does not exist totally umbilical real hypersurface of $Q^{m*}$ ($m\geq3$). Thus,  {\bf Case-1} does not occur. 

In {\bf Case-2}, if $M$ has $\mathfrak{A}$-principal unit normal vector field $N$, then by Theorem \ref{thm:2.7}, $M$ is either an open part of Example \ref{ex:3.1} or Example \ref{ex:3.2} or Example \ref{ex:3.3}. According Examples \ref{ex:3.1}--\ref{ex:3.3} and the assumption
that $M$ has two distinct constant principal curvatures, we
know that $M$ is an open part of Example \ref{ex:3.1}.

In {\bf Case-2}, if $M$ has $\mathfrak{A}$-isotropic unit normal vector field $N$, then by Lemma \ref{lemma:2.6}, we have $S\xi=\alpha\xi$, and $SAN=SA\xi=0$. It follows that $\alpha$ and $0$ are two principal curvatures. In the following, we will divide the discussions into two subcases depending on value of the constant $\alpha$.

{\bf Case-2-i}: $\alpha>0$.

{\bf Case-2-ii}: $\alpha=0$.

In {\bf Case-2-i}, we know that $\alpha$ and $0$ are exactly those two distinct principal curvatures on $M$.
Then, at least one of them belongs in $\sigma(\mathcal{Q})$.

If $\alpha\in\sigma(\mathcal{Q})$, then there exists a tangent vector field $X\in\mathcal{Q}$, such that
$SX=\alpha X$. Now, by Lemma \ref{lemma:5.5}, it follows that $S\phi X=\tfrac{\alpha^2-2}{\alpha}\phi X$, in which $\tfrac{\alpha^2-2}{\alpha}$ is obviously different from $\alpha$. By the assumption
that $M$ has two distinct constant principal curvatures, we have $\tfrac{\alpha^2-2}{\alpha}=0$,
which deduces that $\alpha=\sqrt{2}$ and $0\in\sigma(\mathcal{Q})$. Now, take a tangent vector field $X\in\mathcal{Q}$, such that $SX=0$, by Lemma \ref{lemma:5.5}, it follows that $S\phi X=\sqrt{2}\phi X$.
It follows that $\mathcal{Q}=V_{\sqrt{2}}\oplus V_0$ and
$J V_{\sqrt{2}}=V_0$. Now, by (3) of Lemma \ref{lemma:5.1}, we have
$AV_{\sqrt{2}}\bot{\rm Span}\{V_{\sqrt{2}},JV_{\sqrt{2}}\}$, which contradicts with $AV_{\sqrt{2}}\subset\mathcal{Q}$.

If $0\in\sigma(\mathcal{Q})$, then there exists a tangent vector field $X\in\mathcal{Q}$ such that
$SX=0$. Now, by Lemma \ref{lemma:5.5}, it follows that $S\phi X=\tfrac{2}{\alpha}\phi X$, in which $\tfrac{2}{\alpha}$ is different from $0$. By the assumption that $M$ has two distinct constant principal curvatures, then we have $\tfrac{2}{\alpha}=\alpha$,
which deduces $\alpha=\sqrt{2}$ and $\sqrt{2}\in\sigma(\mathcal{Q})$. Now, taking a tangent vector field $X\in\mathcal{Q}$ such that $SX=\sqrt{2} X$, by Lemma \ref{lemma:5.5}, it follows that $S\phi X=0$.
This implies that $\mathcal{Q}=V_0\oplus V_{\sqrt{2}}$ and $JV_0=V_{\sqrt{2}}$.
But from (3) of Lemma \ref{lemma:5.1}, we have $AV_0\perp {\rm Span}\{V_0,V_{\sqrt{2}}\}$,
which contradicts with $AV_0\subset \mathcal{Q}$.

In {\bf Case-2-ii}, we know that there is another one principal curvature $\lambda\neq0$. If $0\in\sigma(\mathcal{Q})$, then there exists a tangent vector field $X\in\mathcal{Q}$ such that
$SX=0$. Now, by \eqref{eqn:SS1}, it follows that $\phi X=0$,
which is a contradiction. Thus, it follows that $0\notin\sigma(\mathcal{Q})$.
From the assumption that $M$ has two distinct constant principal curvatures,
it holds that
$SX=\lambda X$ for all $X\in\mathcal{Q}$, where $\lambda\neq0$. So $M$ satisfies $S\phi=\phi S$, it follows that $M$ has isometric Reeb flow. According the Theorem \ref{thm:2.2} and the principal curvatures of Example \ref{ex:3.4} and Example \ref{ex:3.5}, we get a contradiction.

We have completed the proof of Theorem \ref{thm:1.1a}.

\subsection{The proof of  Theorem \ref{thm:1.1b}}\label{sect:5.2}~

In this subsection, we assume that $M$ is a Hopf hypersurface of $Q^{m*}$ ($m\geq3$) with three distinct constant principal curvatures. 
If $M$ has $\mathfrak{A}$-principal unit normal vector field $N$, then by Theorem \ref{thm:2.7}, $M$ is either an open part of Example \ref{ex:3.1} or Example \ref{ex:3.2} or Example \ref{ex:3.3}. According Examples \ref{ex:3.1}--\ref{ex:3.3} and the assumption
that $M$ has three distinct constant principal curvatures, we
know that $M$ is either an open part of Example \ref{ex:3.2} or Example \ref{ex:3.3}.

If $M$ has $\mathfrak{A}$-isotropic unit normal vector field $N$, then by Lemma \ref{lemma:2.6}, we have $S\xi=\alpha\xi$, and $SAN=SA\xi=0$. In the following, we will divide the discussions into two cases depending on value of the constant $\alpha$.

{\bf Case-i}: $\alpha>0$.

{\bf Case-ii}: $\alpha=0$.

In {\bf Case-i}, we know that $\alpha$ and $0$ are two distinct principal curvatures on $M$.
In the following, we further divide the discussions into four subcases depending on whether $\alpha$ or $0$
belongs in $\sigma(\mathcal{Q})$.

{\bf Case-i-1}: $\alpha\in\sigma(\mathcal{Q})$ and $0\notin\sigma(\mathcal{Q})$.

{\bf Case-i-2}: $\alpha\notin\sigma(\mathcal{Q})$ and $0\in\sigma(\mathcal{Q})$.

{\bf Case-i-3}: $\alpha,0\in\sigma(\mathcal{Q})$.

{\bf Case-i-4}: $\alpha,0\notin\sigma(\mathcal{Q})$.

In {\bf Case-i-1}, there exists a tangent vector field $X\in\mathcal{Q}$, such that
$SX=\alpha X$. Now, by Lemma \ref{lemma:5.5}, it follows that $S\phi X=\tfrac{\alpha^2-2}{\alpha}\phi X$,
in which $\tfrac{\alpha^2-2}{\alpha}\in\sigma(\mathcal{Q})$ is different from $\alpha$.

If $\alpha\notin \{1,\sqrt{2},2\}$, then $\tfrac{\alpha^2-2}{\alpha}\neq0$, $\tfrac{\alpha^2-2}{\alpha}\neq\pm1$. By $0\notin\sigma(\mathcal{Q})$, we have $\mathcal{Q}=V_\alpha\oplus V_{\tfrac{\alpha^2-2}{\alpha}}$. From \eqref{eqn:SS1}, it can be checked
that $JV_\alpha=V_{\tfrac{\alpha^2-2}{\alpha}}$.
Now, by (3) of Lemma \ref{lemma:5.1}, so
$AV_\alpha\bot{\rm Span}\{V_\alpha,JV_\alpha\}$, which contradicts with $AV_\alpha\subset\mathcal{Q}$.

If $\alpha=1$, then $\tfrac{\alpha^2-2}{\alpha}=-1$. By $0\notin\sigma(\mathcal{Q})$ and Lemma \ref{lemma:5.5}, we can have $\mathcal{Q}=V_1\oplus V_{-1}$ and $JV_1=V_{-1}$.
It follows that $M$ has constant principal curvatures $0,1,-1$ with multiplicities $2, m-1,m-2$.
In particular, according to the principal curvatures of
Example \ref{ex:3.7}, we know that Example \ref{ex:3.7} is contained in this case.

If $\alpha=\sqrt{2}$, then $\tfrac{\alpha^2-2}{\alpha}=0$, which contradicts with $0\notin\sigma(\mathcal{Q})$.

If $\alpha=2$, then $\tfrac{\alpha^2-2}{\alpha}=1$. By $0\notin\sigma(\mathcal{Q})$ and Lemma \ref{lemma:5.5}, we have $\mathcal{Q}=V_2\oplus V_1$ and $JV_2\subset V_1$. By (3) of Lemma \ref{lemma:5.1}, we have
$AV_2\bot{\rm Span}\{V_2,JV_2\}$. Now, we take $X\in V_2$ and $Y=JX\in V_1$ into Cartan's formula
\eqref{eqn:call2g}, we have
$$
0=(2-1)[1+2g(\phi X,JX)^2-3g(AX,X)^2-3g(AX,JX)^2]=3,
$$
which is a contradiction.

In {\bf Case-i-2}, there exists a tangent vector field $X\in\mathcal{Q}$, such that
$SX=0$. Now, by Lemma \ref{lemma:5.5}, it follows that $S\phi X=\tfrac{2}{\alpha}\phi X$, in which $\tfrac{2}{\alpha}\in\sigma(\mathcal{Q})$ is different from $0$.

If $\alpha\notin\{\sqrt{2},2\}$, then $\tfrac{2}{\alpha}\neq\alpha$. By $\alpha\notin\sigma(\mathcal{Q})$ and Lemma \ref{lemma:5.5}, we have $\mathcal{Q}=V_0\oplus V_{\tfrac{2}{\alpha}}$ and
$JV_0=V_{\tfrac{2}{\alpha}}$. Now, by (3) of Lemma \ref{lemma:5.1}, we have
$AV_0\bot{\rm Span}\{V_0,V_{\tfrac{2}{\alpha}}\}$, which contradicts with $AV_0\subset\mathcal{Q}$.

If $\alpha=\sqrt{2}$, then $\tfrac{2}{\alpha}=\sqrt{2}$, which contradicts with $\alpha\notin\sigma(\mathcal{Q})$.

If $\alpha=2$, then $\tfrac{2}{\alpha}=1$. By $\alpha\notin\sigma(\mathcal{Q})$ and Lemma \ref{lemma:5.5}, we have $\mathcal{Q}=V_0\oplus V_1$ and $JV_0\subset V_1$. By (3) of Lemma \ref{lemma:5.1}, we have
$AV_0\bot{\rm Span}\{V_0,JV_0\}$. Now, we take unit vector fields $X\in V_0$ and $Y=JX\in V_1$ into Cartan's formula \eqref{eqn:call2g}, we have
$$
0=(-1)[1+2g(\phi X,JX)^2-3g(AX,JX)^2-3g(AX,X)^2]=-3,
$$
which is a contradiction.

In {\bf Case-i-3}, there exist tangent vector fields $X,Y\in\mathcal{Q}$, such that
$SX=\alpha X$ and $SY=0$. Now, by Lemma \ref{lemma:5.5}, it follows that $S\phi X=\tfrac{\alpha^2-2}{\alpha}\phi X$ and $S\phi Y=\frac{2}{\alpha}\phi Y$,
which implies that $\tfrac{\alpha^2-2}{\alpha},\frac{2}{\alpha}\in \sigma(\mathcal{Q})$.


If $\alpha\notin\{\sqrt{2},2\}$, then $\{\alpha,0,\frac{2}{\alpha},\tfrac{\alpha^2-2}{\alpha}\}$
are four distinct principal curvatures, which is a contradiction.

If $\alpha=\sqrt{2}$, then $\{\alpha,0,\frac{2}{\alpha},\tfrac{\alpha^2-2}{\alpha}\}=\{\sqrt{2},0\}$. So, there is another one constant principal curvature
$\lambda\notin \{\sqrt{2},0\}$ such that $\lambda\in \sigma(\mathcal{Q})$.
Choosing a unit vector field $X\in \mathcal{Q}$ such that $SX=\lambda X$, by Lemma \ref{lemma:5.5}, we have
$S\phi X=\frac{\sqrt{2}\lambda-2}{2\lambda-\sqrt{2}}\phi X$, which implies $\frac{\sqrt{2}\lambda-2}{2\lambda-\sqrt{2}}\in \sigma(\mathcal{Q})$. It also follows
from $\lambda\notin \{\sqrt{2},0\}$ that $\lambda\neq\frac{\sqrt{2}\lambda-2}{2\lambda-\sqrt{2}}$ and $\frac{\sqrt{2}\lambda-2}{2\lambda-\sqrt{2}}\notin\{\sqrt{2},0\}$. Thus, there are four distinct constant principal curvatures $\{\sqrt{2},0,\lambda,\frac{\sqrt{2}\lambda-2}{2\lambda-\sqrt{2}}\}\in \sigma(\mathcal{Q})$, which is a contradiction.

If $\alpha=2$, then $\{\alpha,0,\frac{2}{\alpha},\tfrac{\alpha^2-2}{\alpha}\}=\{2,0,1\}$.
By Lemma \ref{lemma:5.5}, we can have $\mathcal{Q}=V_2\oplus V_0\oplus V_1$, $JV_2\subset V_1$
and $JV_0\subset V_1$.
By (3) of Lemma \ref{lemma:5.1}, on $\mathcal{Q}$, it holds
$AV_2\bot {\rm Span}\{V_2, JV_2\}$ and $AV_0\bot {\rm Span}\{V_0, JV_0\}$.
Now, we take unit vector fields $X\in V_0$ and $Y=JX\in V_1$ into \eqref{eqn:call}, we have
\begin{equation*}
0\leq2\sum_{\mu_i=2}\frac{g((\nabla_{e_i} S)X,JX)^2}{(0-2)(1-2)}
=(-1+0\times1)\Big\{1+2g(\phi X,JX)^2-3g(AX,JX)^2-3g(AX,X)^2\Big\}=-3,
\end{equation*}
which is a contradiction.

In {\bf Case-i-4}, we know that there is another one constant principal curvature $\lambda\notin\{\alpha,0\}$ which belongs in $\sigma(\mathcal{Q})$. By
the assumption that $M$ has three distinct constant principal curvatures and
$\alpha,0\notin\sigma(\mathcal{Q})$, we have $SX=\lambda X$ for all $X\in\mathcal{Q}$. It follows  that $M$ satisfies $S\phi=\phi S$, and $M$ has isometric Reeb flow. According the Theorem \ref{thm:2.2} and the principal curvatures of Example \ref{ex:3.4} and Example \ref{ex:3.5}, we know
that $M$ is an open part of Example \ref{ex:3.4}.

In {\bf Case-ii}, it holds $\alpha=0$. If $0\in\sigma(\mathcal{Q})$, then there exists a unit tangent vector field $X\in\mathcal{Q}$ such that $SX=0$. Now, by Lemma \ref{lemma:5.5}, it follows that $\phi X=0$, which is a contradiction. So, $0\notin\sigma(\mathcal{Q})$.
Then, there are another two distinct constant principal curvatures $\{\lambda,\mu\}=\sigma(\mathcal{Q})$, which holds $\lambda,\mu\neq0$.
By Lemma \ref{lemma:5.5}, we have
$\mu=\tfrac{-1}{\lambda}$, and it holds $\mathcal{Q}=V_\lambda\oplus V_\mu$ and $JV_\lambda=V_\mu$.

If $\lambda\notin\{1,-1\}$, then by (3) of Lemma \ref{lemma:5.1}, we have
$AV_\lambda\bot{\rm Span}\{V_\lambda,JV_\lambda\}$, which contradicts with $AV_\lambda\subset\mathcal{Q}$. Thus, without loss of generality, it must hold $\lambda=-1$ and $\mu=1$.
Now, we have $\alpha=0$, $\mathcal{Q}=V_1\oplus V_{-1}$, $JV_1=V_{-1}$ on $M$. It follows that $M$ has constant principal curvatures $0,1,-1$ with multiplicities $3, m-2,m-2$.
In particular, according to the principal curvatures of
Example \ref{ex:3.6}, we know that Example \ref{ex:3.6} is contained in this case.

We have completed the proof of Theorem \ref{thm:1.1b}.

\subsection{The proof of Theorem \ref{thm:1.1c}}\label{sect:5.3}~

In this subsection, we assume that $M$ is a Hopf hypersurface of $Q^{m*}$ ($m\geq3$) with four distinct constant principal curvatures. 
If $M$ has $\mathfrak{A}$-principal unit normal vector field $N$, then by Theorem \ref{thm:2.7}, $M$ is either an open part of Example \ref{ex:3.1} or Example \ref{ex:3.2} or Example \ref{ex:3.3}. According Examples \ref{ex:3.1}--\ref{ex:3.3} and the assumption
that $M$ has four distinct constant principal curvatures, we get a contradiction.

If $M$ has $\mathfrak{A}$-isotropic unit normal vector field $N$, then by Lemma \ref{lemma:2.6}, we have $S\xi=\alpha\xi$, and $SAN=SA\xi=0$. In the following,
we still divide the discussions into two cases depending on the value of $\alpha$.

{\bf Case-i}: $\alpha>0$.

{\bf Case-ii}: $\alpha=0$.

In {\bf Case-i}, $\alpha$ and $0$ are two distinct principal curvatures on $M$.
Then we further divide the discussions into four subcases depending on whether $\alpha$ or $0$
belongs in $\sigma(\mathcal{Q})$.

{\bf Case-i-1}: $\alpha\in\sigma(\mathcal{Q})$ and $0\notin\sigma(\mathcal{Q})$.

{\bf Case-i-2}: $\alpha\notin\sigma(\mathcal{Q})$ and $0\in\sigma(\mathcal{Q})$.

{\bf Case-i-3}: $\alpha,0\in\sigma(\mathcal{Q})$.

{\bf Case-i-4}: $\alpha,0\notin\sigma(\mathcal{Q})$.

In {\bf Case-i-1}, there exists a tangent vector field $X\in\mathcal{Q}$, such that
$SX=\alpha X$. Now, by Lemma \ref{lemma:5.5}, it follows that $S\phi X=\tfrac{\alpha^2-2}{\alpha}\phi X$,
in which $\tfrac{\alpha^2-2}{\alpha}\in\sigma(\mathcal{Q})$ is different from $\alpha$.

If $\alpha\notin\{\sqrt{2},2\}$, then there is another one constant principal curvature $\lambda\in\sigma(\mathcal{Q})$ and
$\lambda\notin\{\alpha,0,\tfrac{\alpha^2-2}{\alpha}\}$. For any $X\in V_\lambda$, by Lemma \ref{lemma:5.5}, we have $(2\lambda-\alpha)S\phi X=(\alpha\lambda-2)\phi X$. It follows from $\alpha\neq2$, $0\notin\sigma(\mathcal{Q})$ and $\lambda\notin\{\alpha,0,\tfrac{\alpha^2-2}{\alpha}\}$ that $\frac{\alpha\lambda-2}{2\lambda-\alpha}\in\sigma(\mathcal{Q})$ and
$\frac{\alpha\lambda-2}{2\lambda-\alpha}$
is different from $\{\alpha,0,\tfrac{\alpha^2-2}{\alpha}\}$. According to the assumption that
$M$ has four distinct principal curvatures, then we have $\lambda=\frac{\alpha\lambda-2}{2\lambda-\alpha}$, which deduces that
$\alpha>2$ and $\lambda=\frac{\alpha\pm\sqrt{\alpha^2-4}}{2}$. Now, we have
$$
\begin{aligned}
\mathcal{Q}=V_\alpha\oplus V_{\tfrac{\alpha^2-2}{\alpha}}\oplus V_{\lambda}, \ \ \phi V_\alpha= V_{\tfrac{\alpha^2-2}{\alpha}},\ \ \phi V_\lambda=V_\lambda.
\end{aligned}
$$
Moreover, by (3) of Lemma \ref{lemma:5.1} and $\alpha>2$, it holds $AV_{\alpha}\perp \{V_{\alpha},JV_{\alpha}\}$ and $AV_{\alpha}\subset V_\lambda$.

If $\lambda=\frac{\alpha-\sqrt{\alpha^2-4}}{2}$,
by $\alpha>2$, we have $0<\frac{\alpha-\sqrt{\alpha^2-4}}{2}<\tfrac{\alpha^2-2}{\alpha}<\frac{\alpha+\sqrt{\alpha^2-4}}{2}
<\alpha$, then we have $\lambda_+=\alpha$, and $M_+$ is a focal submanifold of $M$. By Proposition \ref{prop:4.3w}, we have that the distinct constant principal curvatures of $M_+$ are
$$
\frac{\alpha^2-2-\alpha\sqrt{\alpha^2-4}}{\alpha+\sqrt{\alpha^2-4}}, \ \
\frac{\alpha(\alpha^2-3)}{2}, \ \ \alpha(\alpha^2-3),\ \ 0.
$$
By $\alpha>2$, we have
$$
0<\frac{\alpha^2-2-\alpha\sqrt{\alpha^2-4}}{\alpha+\sqrt{\alpha^2-4}}
<\frac{\alpha(\alpha^2-3)}{2}<\alpha(\alpha^2-3).
$$
Thus, it contradicts with the fact that $M_+$ is austere.

If $\lambda=\frac{\alpha+\sqrt{\alpha^2-4}}{2}$,
we take $X\in V_\alpha$, $Y=AX\in V_{\frac{\alpha+\sqrt{\alpha^2-4}}{2}}$ into \eqref{eqn:call}, we have
\begin{equation}\label{eqn:5.2sq}
2\sum_{\mu_i=\tfrac{\alpha^2-2}{\alpha}}\frac{g((\nabla_{e_i} S)X,AX)^2}{(\alpha-\tfrac{\alpha^2-2}{\alpha})(\frac{\alpha+\sqrt{\alpha^2-4}}{2}-\tfrac{\alpha^2-2}{\alpha})}
=(-1+\alpha\frac{\alpha+\sqrt{\alpha^2-4}}{2})(1-2).
\end{equation}
By $\alpha>2$, we have $0<\frac{\alpha-\sqrt{\alpha^2-4}}{2}<\tfrac{\alpha^2-2}{\alpha}<\frac{\alpha+\sqrt{\alpha^2-4}}{2}
<\alpha$ and $\frac{\alpha(\alpha+\sqrt{\alpha^2-4})}{2}>1$, it implies that the left side of \eqref{eqn:5.2sq} is non-negative, the right side of \eqref{eqn:5.2sq} is negative, which is a contradiction.

If $\alpha=\sqrt{2}$, then $\tfrac{\alpha^2-2}{\alpha}=0$, which contradicts with $0\notin\sigma(\mathcal{Q})$.

If $\alpha=2$, then $\tfrac{\alpha^2-2}{\alpha}=1$. By the assumption that $M$ has four distinct principal curvatures, we have another one constant principal curvature $\lambda\in\sigma(\mathcal{Q})$ and $\lambda\notin\{2,0,1\}$. For any $X\in V_\lambda$, by Lemma \ref{lemma:5.5}, we have $S\phi X=\phi X$.
Now, we have
$$
\begin{aligned}
\mathcal{Q}=V_2\oplus V_\lambda\oplus V_1, \ \ \phi V_2\subset V_1,\ \ \phi V_\lambda\subset V_1.
\end{aligned}
$$
If $\lambda<1$ and $\lambda\neq-1$, then by (3) of Lemma \ref{lemma:5.1}, it holds
$AV_\lambda\perp \{V_\lambda,JV_\lambda\}$. Now, we take unit vector fields $X\in V_\lambda$ and $Y=JX\in V_1$ into \eqref{eqn:call}, we have
$$
0\leq2\sum_{\mu_i=2}\frac{g((\nabla_{e_i} S)X,JX)^2}{(1-2)(\lambda-2)}
=3(-1+\lambda)<0,
$$
which is a contradiction.
If $\lambda=-1$, then $\lambda_+=2$, and $M_+$ is a focal submanifold of $M$. By Proposition \ref{prop:4.3w}, we have that the distinct constant principal curvatures of $M_+$ are $2,0,1,-1$, which deduces that
$M_+$ is not austere. It is a contradiction.
If $1<\lambda<2$, then by (3) of Lemma \ref{lemma:5.1}, it holds
$AV_2\perp \{V_2,JV_2\}$. Now, we take unit vector fields $X\in V_2$ and $Y=JX\in V_1$ into \eqref{eqn:call}, we have
$$
0\geq2\sum_{\mu_i=\lambda}\frac{g((\nabla_{e_i} S)X,JX)^2}{(2-\lambda)(1-\lambda)}
=3>0,
$$
which is a contradiction. If $\lambda>2$, then by (3) of Lemma \ref{lemma:5.1}, it still holds
$AV_\lambda\perp \{V_\lambda,JV_\lambda\}$. Now, we take unit vector fields $X\in V_\lambda$ and $Y=JX\in V_1$ into \eqref{eqn:call}, we have
$$
0\geq2\sum_{\mu_i=2}\frac{g((\nabla_{e_i} S)X,JX)^2}{(\lambda-2)(1-2)}
=3(\lambda-1)>0,
$$
which is a contradiction.

In {\bf Case-i-2}, there exists a tangent vector field $X\in\mathcal{Q}$, such that
$SX=0$. Now, by Lemma \ref{lemma:5.5}, it follows that $S\phi X=\tfrac{2}{\alpha}\phi X$, in which $\tfrac{2}{\alpha}\in\sigma(\mathcal{Q})$ is different from $0$.

If $\alpha\notin\{\sqrt{2},2\}$, then there is $\lambda\in\sigma(\mathcal{Q})$ and
$\lambda\notin\{\alpha,0,\tfrac{2}{\alpha}\}$. For any $X\in V_\lambda$, by Lemma \ref{lemma:5.5}, we have $(2\lambda-\alpha)S\phi X=(\alpha\lambda-2)\phi X$. It follows from $\alpha\neq2$,
$\alpha\notin\sigma(\mathcal{Q})$ and $\lambda\notin\{\alpha,0,\tfrac{2}{\alpha}\}$ that $\frac{\alpha\lambda-2}{2\lambda-\alpha}\in\sigma(\mathcal{Q})$ and $\frac{\alpha\lambda-2}{2\lambda-\alpha}$
is different from $\{\alpha,0,\tfrac{2}{\alpha}\}$. According to the assumption that
$M$ has four distinct principal curvatures, then we have $\lambda=\frac{\alpha\lambda-2}{2\lambda-\alpha}$, which deduces that
$\alpha>2$ and $\lambda=\frac{\alpha\pm\sqrt{\alpha^2-4}}{2}$. Note that
by $\alpha>2$, we have $\frac{\alpha\pm\sqrt{\alpha^2-4}}{2}\neq1$.
Now, we have
$$
\begin{aligned}
\mathcal{Q}=V_0\oplus V_{\tfrac{2}{\alpha}}\oplus V_{\lambda}, \ \ \phi V_0= V_{\tfrac{2}{\alpha}},\ \ \phi V_\lambda=V_\lambda.
\end{aligned}
$$
Moreover, by (3) of Lemma \ref{lemma:5.1}, it holds $AV_{0}\perp \{V_{0},JV_{0}\}$
and $A(V_{0}\oplus V_{\tfrac{2}{\alpha}})=V_\lambda$.

If $\lambda=\frac{\alpha-\sqrt{\alpha^2-4}}{2}$,
we have $0<\frac{\alpha-\sqrt{\alpha^2-4}}{2}<\frac{2}{\alpha}$.
Assume that ${\rm dim}V_0\geq2$, then one can choose some unit vector fields $X\in V_0$ and $Y\in V_{\frac{\alpha-\sqrt{\alpha^2-4}}{2}}$ such that $AX\perp \{Y,JY\}$. Now, we take such $(X,Y)$
into \eqref{eqn:call}, we have
\begin{equation*}
\begin{aligned}
0\leq2\sum_{\mu_i=\frac{2}{\alpha}}\frac{g((\nabla_{e_i} S)X,Y)^2}{(0-\frac{2}{\alpha})(\frac{\alpha-\sqrt{\alpha^2-4}}{2}-\frac{2}{\alpha})}
=&(-1)\Big\{1+2g(\phi X,Y)^2-2g(AX,Y)^2-2g(AX,JY)^2\\
&  +g(AX,X)g(AY,Y)+g(AX,JX)g(AY,JY)\Big\}=-1,
\end{aligned}
\end{equation*}
which is a contradiction. So ${\rm dim}V_0=1$, then we can choose a frame $\{e_1,e_2=Je_1,e_3=Ae_1,e_4=Je_3,e_5=AN,e_6=A\xi,e_7=\xi\}$ such that $e_1\in V_0,e_2\in V_{\frac{2}{\alpha}},\{e_3,e_4\}\in V_{\frac{\alpha-\sqrt{\alpha^2-4}}{2}}$.
On one hand, by (1) of Lemma \ref{lemma:5.1}, we calculate
$$
\begin{aligned}
(\nabla_{e_1} S){e_5}-(\nabla_{e_5} S){e_1}=-S(\nabla_{e_1}{e_5})+S\nabla_{e_5}{e_1}
=-S(q(e_1)A\xi)+S\nabla_{e_5}{e_1}=S\nabla_{e_5}{e_1}.
\end{aligned}
$$
By using Codazzi equation \eqref{eqn:2.9}, we have
$$
\begin{aligned}
&(\nabla_{e_1} S){e_5}-(\nabla_{e_5} S){e_1}=e_3.
\end{aligned}
$$
It follows that
$g(\nabla_{e_5}{e_1},e_3)=\frac{2}{\alpha-\sqrt{\alpha^2-4}}$.
On the other hand, we calculate
$$
\begin{aligned}
g(\nabla_{e_5}{e_3},e_1)=g(\nabla_{e_5}(Ae_1),e_1)=g(q(e_5)JAe_1+A\nabla_{e_5}{e_1},e_1)
=g(\nabla_{e_5}{e_1},e_3),
\end{aligned}
$$
which deduces that $g(\nabla_{e_5}{e_1},e_3)=0$. Then $\frac{2}{\alpha-\sqrt{\alpha^2-4}}=0$,
we get a contradiction.

If $\lambda=\frac{\alpha+\sqrt{\alpha^2-4}}{2}$,
we take unit vector fields $X\in V_0$, $Y=AX\in V_{\frac{\alpha+\sqrt{\alpha^2-4}}{2}}$ into \eqref{eqn:call}, we have
\begin{equation}\label{eqn:5.4sq}
2\sum_{\mu_i=\tfrac{2}{\alpha}}\frac{g((\nabla_{e_i} S)X,AX)^2}{(0-\tfrac{2}{\alpha})(\frac{\alpha+\sqrt{\alpha^2-4}}{2}-\tfrac{2}{\alpha})}
=(-1)(1-2)=1.
\end{equation}
By $\alpha>2$, we have $0<\frac{\alpha-\sqrt{\alpha^2-4}}{2}<\tfrac{2}{\alpha}<\frac{\alpha+\sqrt{\alpha^2-4}}{2}$, it implies that the left side of \eqref{eqn:5.4sq} is non-positive, the right side of \eqref{eqn:5.4sq} is positive, which is a contradiction.

If $\alpha=\sqrt{2}$, then $\tfrac{2}{\alpha}=\sqrt{2}$, which contradict with $\alpha\notin\sigma(\mathcal{Q})$.

If $\alpha=2$, then $\tfrac{2}{\alpha}=1$. By the assumption that $M$ has four distinct principal curvatures, we have another one constant principal curvature $\lambda\in\sigma(\mathcal{Q})$ and $\lambda\notin\{2,0,1\}$. Then, by Lemma \ref{lemma:5.5}, we can have
$$
\begin{aligned}
\mathcal{Q}=V_0\oplus V_\lambda\oplus V_1, \ \ \phi V_0\subset V_1,\ \ \phi V_\lambda\subset V_1.
\end{aligned}
$$
If $\lambda<0$ or $\lambda>1$, by (3) of Lemma \ref{lemma:5.1}, it holds
$AV_0\perp \{V_0,JV_0\}$. Now, we take unit vector fields $X\in V_0$ and $Y=JX\in V_1$ into \eqref{eqn:call}, we have
$$
0\leq2\sum_{\mu_i=\lambda}\frac{g((\nabla_{e_i} S)X,JX)^2}{(0-\lambda)(1-\lambda)}
=-3<0,
$$
which is a contradiction.
If $0<\lambda<1$, by (3) of Lemma \ref{lemma:5.1}, it holds
$AV_\lambda\perp \{V_\lambda,JV_\lambda\}$. Now, we take unit vector fields $X\in V_\lambda$ and $Y=JX\in V_1$ into \eqref{eqn:call}, we have
$$
0\leq2\sum_{\mu_i=0}\frac{g((\nabla_{e_i} S)X,JX)^2}{(\lambda-0)(1-0)}
=3(\lambda-1)<0,
$$
which is a contradiction.

In {\bf Case-i-3}, there exist tangent vector fields $X,Y\in\mathcal{Q}$, such that
$SX=\alpha X$ and $SY=0$. Now, by Lemma \ref{lemma:5.5}, it follows that $S\phi X=\tfrac{\alpha^2-2}{\alpha}\phi X$ and $S\phi Y=\frac{2}{\alpha}\phi Y$,
which implies that $\tfrac{\alpha^2-2}{\alpha},\frac{2}{\alpha}\in \sigma(\mathcal{Q})$.

When $\alpha\notin\{\sqrt{2},2\}$, then $\{\alpha,0,\tfrac{\alpha^2-2}{\alpha},\frac{2}{\alpha}\}$
are four distinct constant principal curvatures. By Lemma \ref{lemma:5.5}, we have
$$
\begin{aligned}
\mathcal{Q}=V_\alpha\oplus V_{\frac{\alpha^2-2}{\alpha}}\oplus V_0\oplus V_{\frac{2}{\alpha}}, \ \ \phi V_\alpha= V_{\frac{\alpha^2-2}{\alpha}},\ \ \phi V_0= V_{\frac{2}{\alpha}}.
\end{aligned}
$$

If $\alpha>2$, then $0<\frac{2}{\alpha}<\frac{\alpha^2-2}{\alpha}<\alpha$. By (3) of Lemma \ref{lemma:5.1}, it holds
$AV_0\perp \{V_0,JV_0\}$. Now, we take unit vector fields $X\in V_0$ and $Y=JX\in V_{\frac{2}{\alpha}}$ into \eqref{eqn:call}, we have
$$
0\leq2\sum_{\mu_i=\alpha}\frac{g((\nabla_{e_i} S)X,JX)^2}{(0-\alpha)(\frac{2}{\alpha}-\alpha)}+2\sum_{\mu_i=\frac{\alpha^2-2}{\alpha}}\frac{g((\nabla_{e_i} S)X,JX)^2}{(0-\frac{\alpha^2-2}{\alpha})(\frac{2}{\alpha}-\frac{\alpha^2-2}{\alpha})}
=-3<0,
$$
which is a contradiction.

If $\sqrt{2}<\alpha<2$, then $0<\frac{\alpha^2-2}{\alpha}<\frac{2}{\alpha}<\alpha$. So $\lambda_+=\alpha$
and $M_+$ is a focal submanifold of $M$. By Proposition \ref{prop:4.3w}, $M_+$ has
constant principal curvatures:
$0,\frac{\alpha(\alpha^2-3)}{2},-\frac{1}{\alpha},\frac{\alpha}{\alpha^2-2},\alpha(\alpha^2-3)$.
And it follows from
Proposition \ref{prop:4.3w} that $-\frac{1}{\alpha}<\frac{\alpha(\alpha^2-3)}{2}<\frac{\alpha}{\alpha^2-2}$.
Furthermore, we can know that $\alpha(\alpha^2-3)$ is different from $\frac{\alpha(\alpha^2-3)}{2},-\frac{1}{\alpha},\frac{\alpha}{\alpha^2-2}$. In
fact, if $\alpha(\alpha^2-3)=\frac{-1}{\alpha}$, by $\sqrt{2}<\alpha<2$, we have $\alpha=\frac{1+\sqrt{5}}{2}$. It follows that $M_+$
has constant principal curvatures $\frac{1-\sqrt{5}}{2},\frac{1-\sqrt{5}}{4},\frac{3+\sqrt{5}}{2},0$, which contradicts with the fact that $M_+$ is austere.
If $\alpha(\alpha^2-3)=\frac{\alpha(\alpha^2-3)}{2}$, by $\sqrt{2}<\alpha<2$, we have $\alpha=\sqrt{3}$. It follows that $M_+$
has constant principal curvatures $\frac{-1}{\sqrt{3}},\sqrt{3},0$, which contradicts with the fact that $M_+$ is austere. If $\alpha(\alpha^2-3)=\frac{\alpha}{\alpha^2-2}$, by $\sqrt{2}<\alpha<2$, we have $\alpha=\sqrt{\frac{5+\sqrt{5}}{2}}$. It follows that $M_+$
has constant principal curvatures $-\sqrt{\frac{2}{5+\sqrt{5}}},\frac{(\sqrt{5}-1)\sqrt{5+\sqrt{5}}}{4\sqrt{2}},
\frac{\sqrt{2(5+\sqrt{5})}}{1+\sqrt{5}},0$, which still contradicts with the fact that $M_+$ is austere.
Now, if one of $\{\frac{\alpha(\alpha^2-3)}{2},-\frac{1}{\alpha},\frac{\alpha}{\alpha^2-2},\alpha(\alpha^2-3)\}$ is $0$, then there are three distinct nonzero principal curvatures on $M_+$, which
contradicts with the fact that $M_+$ is austere. So $\{\frac{\alpha(\alpha^2-3)}{2},-\frac{1}{\alpha},\frac{\alpha}{\alpha^2-2},\alpha(\alpha^2-3)\}$ are all nonzero. Then due that $M_+$ is austere, it must hold
$$
\frac{\alpha(\alpha^2-3)}{2}-\frac{1}{\alpha}+\frac{\alpha}{\alpha^2-2}+\alpha(\alpha^2-3)=0.
$$
But above equation has no real solution to $\alpha$. It is a contradiction.

If $0<\alpha<\sqrt{2}$, then $\frac{\alpha^2-2}{\alpha}<0<\alpha<\frac{2}{\alpha}$. So  $\lambda_+=\frac{2}{\alpha}$
and $M_+$ is a focal submanifold of $M$. By Proposition \ref{prop:4.3w}, $M_+$ has constant principal curvatures:
$0,\frac{-\alpha}{2},-\alpha,\frac{-1}{\alpha},\frac{\alpha}{2-\alpha^2}$. And it follows from
$0<\alpha<\sqrt{2}$ that $\frac{-1}{\alpha}<\frac{-\alpha}{2}<\frac{\alpha}{2-\alpha^2}$.
Furthermore, we can know that $-\alpha$ is different from $\frac{-1}{\alpha},\frac{-\alpha}{2},\frac{\alpha}{2-\alpha^2}$. In
fact, by $0<\alpha<\sqrt{2}$, we only need to consider $-\alpha=\frac{-1}{\alpha}$, which deduces that $\alpha=1$. If $\alpha=1$, it follows that $M_+$
has constant principal curvatures $-1,\frac{-1}{2},1,0$, which contradicts with the fact that $M_+$ is austere.
Now, according to that $\{\frac{-\alpha}{2},-\alpha,\frac{-1}{\alpha},\frac{\alpha}{2-\alpha^2}\}$ are
nonzero, and they are different to each other, then from the fact that $M_+$ is austere, it must hold
$$
-\frac{\alpha}{2}-\alpha-\frac{1}{\alpha}+\frac{\alpha}{2-\alpha^2}=0,
$$
which deduces $\alpha=\sqrt{\frac{\sqrt{13}+1}{3}}$. But if $\alpha=\sqrt{\frac{\sqrt{13}+1}{3}}$,
by Proposition \ref{prop:4.3w},
$M_+$ has constant principal curvatures:
$-\frac{\sqrt{\sqrt{13}-1}}{2}$, $-\frac{1}{2}\sqrt{\frac{\sqrt{13}+1}{3}}$, $\sqrt{\frac{7}{2}+\sqrt{13}}$, $-\sqrt{\frac{\sqrt{13}+1}{3}}$ and $0$,
which means that $M_+$ is not austere. We get a contradiction.

If $\alpha=\sqrt{2}$, then there is another principal curvature $\lambda\in\sigma(\mathcal{Q})$ and
$\lambda\notin\{\sqrt{2},0\}$. For any $X\in V_\lambda$, by Lemma \ref{lemma:5.5}, we have $(2\lambda-\sqrt{2})S\phi X=(\sqrt{2}\lambda-2)\phi X$. It follows that $\mu=\frac{\sqrt{2}\lambda-2}{2\lambda-\sqrt{2}}\in\sigma(\mathcal{Q})$, which
is also different from $\{0,\sqrt{2}\}$. From the relationship between $\lambda$ and
$\mu$, if $\lambda\in(-\infty,0)$ then $\mu\in(\frac{1}{\sqrt{2}},\sqrt{2})$.
If $\lambda\in(\sqrt{2},+\infty)$, then $\mu\in(0,\frac{1}{\sqrt{2}})$.
So, by exchange $\lambda$ and $\mu$ if necessary, we only need to consider two cases that
$\lambda<0<\mu<\sqrt{2}$ or $0<\mu<\sqrt{2}<\lambda$.

If $\lambda<0<\mu<\sqrt{2}$, then $\lambda_+=\sqrt{2}$, and $M_+$
is a focal submanifold of $M$. By Proposition \ref{prop:4.3w}, $M_+$ has constant principal curvatures:
$0,-\sqrt{2},\frac{\sqrt{2}\lambda-1}{\sqrt{2}-\lambda},\frac{-1}{\lambda},
-\frac{1}{\sqrt{2}}$. By $\lambda<0<\mu<\sqrt{2}$, we have $-\sqrt{2}<\frac{\sqrt{2}\lambda-1}{\sqrt{2}-\lambda}<\frac{-1}{\sqrt{2}}<0<\frac{-1}{\lambda}$.
It contradicts with the fact that $M_+$ is austere.

If $0<\mu<\sqrt{2}<\lambda$, then $\lambda_+=\lambda$, and $M_+$ is a focal submanifold of $M$.
Now, $M_+$ has constant principal curvatures:
$0,\sqrt{2}+\frac{-2\lambda}{\lambda^2-\sqrt{2}\lambda+1},\frac{\lambda\mu-1}{\lambda-\mu},
\frac{\sqrt{2}\lambda-1}{\lambda-\sqrt{2}},\frac{-1}{\lambda}$, where $\mu=\frac{\sqrt{2}\lambda-2}{2\lambda-\sqrt{2}}$.
By Proposition \ref{prop:4.3w}, it holds  $\frac{-1}{\lambda}<\frac{\lambda\mu-1}{\lambda-\mu}<\frac{\sqrt{2}\lambda-1}{\lambda-\sqrt{2}}$.
Furthermore, we can know that $\sqrt{2}+\frac{-2\lambda}{\lambda^2-\sqrt{2}\lambda+1}$ is different from $\frac{\lambda\mu-1}{\lambda-\mu},
\frac{\sqrt{2}\lambda-1}{\lambda-\sqrt{2}},\frac{-1}{\lambda}$. In
fact, if $\sqrt{2}+\frac{-2\lambda}{\lambda^2-\sqrt{2}\lambda+1}=\frac{-1}{\lambda}$, by $\lambda>\sqrt{2}$, we have $\lambda=\frac{\sqrt{2}+\sqrt{6}}{2}$. It follows that $M_+$
has constant principal curvatures $\frac{-2}{\sqrt{2}+\sqrt{6}},\frac{-3+\sqrt{3}}{2\sqrt{6}},
\sqrt{3(2+\sqrt{3})},0$, which contradicts with the fact that $M_+$ is austere.
If $\sqrt{2}+\frac{-2\lambda}{\lambda^2-\sqrt{2}\lambda+1}=\frac{\lambda\mu-1}{\lambda-\mu}$, by $\lambda>\sqrt{2}$, we have $\lambda=1+\sqrt{2}$. It follows that $M_+$
has constant principal curvatures $1-\sqrt{2},1+\sqrt{2},0$, which contradicts with the fact that $M_+$ is austere. If $\sqrt{2}+\frac{-2\lambda}{\lambda^2-\sqrt{2}\lambda+1}=\frac{\sqrt{2}\lambda-1}{\lambda-\sqrt{2}}$, by $\lambda>\sqrt{2}$, there is no solution to $\lambda$.
Now, according to that $\{\sqrt{2}+\frac{-2\lambda}{\lambda^2-\sqrt{2}\lambda+1},\frac{\lambda\mu-1}{\lambda-\mu},
\frac{\sqrt{2}\lambda-1}{\lambda-\sqrt{2}},\frac{-1}{\lambda}\}$ are different to each other, and the fact that $M_+$ is austere, it must hold
$$
\sqrt{2}+\frac{-2\lambda}{\lambda^2-\sqrt{2}\lambda+1}+\frac{\lambda\mu-1}{\lambda-\mu}+
\frac{\sqrt{2}\lambda-1}{\lambda-\sqrt{2}}+\frac{-1}{\lambda}=0.
$$
From $\mu=\frac{\sqrt{2}\lambda-2}{2\lambda-\sqrt{2}}$, above equation does not have a real solution to $\lambda$. Thus, we get a contradiction.

If $\alpha=2$, then $\tfrac{\alpha^2-2}{\alpha}=\frac{2}{\alpha}=1$ and there is another one principal curvature $\lambda\in\sigma(\mathcal{Q})$ and
$\lambda\notin\{2,0,1\}$. Now, by Lemma \ref{lemma:5.5}, we have
$$
\begin{aligned}
\mathcal{Q}=V_2\oplus V_0\oplus V_1\oplus V_\lambda, \ \ \phi V_2\subset V_1,\ \ \phi V_0\subset V_1,\ \ \phi V_\lambda\subset V_1.
\end{aligned}
$$
If $\lambda<0<1<2$ or $0<1<\lambda<2$ or $0<1<2<\lambda$, by (3) of Lemma \ref{lemma:5.1}, it holds
$AV_0\perp \{V_0,JV_0\}$. Now, we take
unit vector fields $X\in V_0$ and $Y=JX\in V_1$ into \eqref{eqn:call}, we have
$$
0\leq2\sum_{\mu_i=2}\frac{g((\nabla_{e_i} S)X,JX)^2}{(0-2)(1-2)}+2\sum_{\mu_i=\lambda}\frac{g((\nabla_{e_i} S)X,JX)^2}{(0-\lambda)(1-\lambda)}
=-3<0,
$$
which is a contradiction.
If $0<\lambda<1<2$, by (3) of Lemma \ref{lemma:5.1}, it holds
$AV_\lambda\perp \{V_\lambda,JV_\lambda\}$. Now, we take unit vector fields $X\in V_\lambda$ and $Y=JX\in V_1$ into \eqref{eqn:call}, we have
$$
0\leq2\sum_{\mu_i=2}\frac{g((\nabla_{e_i} S)X,JX)^2}{(\lambda-2)(1-2)}+2\sum_{\mu_i=0}\frac{g((\nabla_{e_i} S)X,JX)^2}{(\lambda-0)(1-0)}
=3(\lambda-1)<0,
$$
which is a contradiction.

In {\bf Case-i-4}, we know that there is another one constant principal curvature
$\lambda\in\sigma(\mathcal{Q})$, and it holds
$\lambda\notin\{\alpha,0\}$. For any $X\in V_\lambda$, by Lemma \ref{lemma:5.5}, we have $(2\lambda-\alpha)S\phi X=(\alpha\lambda-2)\phi X$.

If $0<\alpha<2$, then we have $\frac{\alpha\lambda-2}{2\lambda-\alpha}\in\sigma(\mathcal{Q})$
and $\frac{\alpha\lambda-2}{2\lambda-\alpha}\neq\lambda$.
So $\mathcal{Q}=V_{\lambda}\oplus V_{\frac{\alpha\lambda-2}{2\lambda-\alpha}}$.
If $\lambda\neq\pm1$, by (3) of Lemma \ref{lemma:5.1}, on $\mathcal{Q}$, it holds
$AV_\lambda\bot {\rm Span}\{V_\lambda, V_{\frac{\alpha\lambda-2}{2\lambda-\alpha}}\}$, which is
a contradiction. Thus, it holds $\{\lambda,\frac{\alpha\lambda-2}{2\lambda-\alpha}\}=\{-1,1\}$.
Now, we have
$$
\begin{aligned}
\mathcal{Q}=V_1\oplus V_{-1}, \ \ \phi V_1= V_{-1},\ \ {\rm dim}V_1={\rm dim}V_{-1}=m-2.
\end{aligned}
$$
In particular, according to the principal curvatures of
Example \ref{ex:3.8}, we know that Example \ref{ex:3.8} is contained in this case.

If $\alpha>2$, then we have $\frac{\alpha\lambda-2}{2\lambda-\alpha}\in\sigma(\mathcal{Q})$.
We assume that $\lambda\neq\pm1$ and $\lambda\neq\frac{\alpha\lambda-2}{2\lambda-\alpha}$, by Lemma \ref{lemma:5.5}, we know that
$$
\begin{aligned}
\mathcal{Q}=V_\lambda\oplus V_{\frac{\alpha\lambda-2}{2\lambda-\alpha}}, \ \
\phi V_\lambda= V_{\frac{\alpha\lambda-2}{2\lambda-\alpha}}.
\end{aligned}
$$
But, by (3) of Lemma \ref{lemma:5.1}, it holds
$AV_\lambda\bot {\rm Span}\{V_\lambda, V_{\frac{\alpha\lambda-2}{2\lambda-\alpha}}\}$, which is
a contradiction. So, we have $\lambda=\pm1$ or $\lambda=\frac{\alpha\lambda-2}{2\lambda-\alpha}$.

If $\lambda=\pm1$, then $\frac{\alpha\lambda-2}{2\lambda-\alpha}=\mp1$.
By Lemma \ref{lemma:5.5}, we know that
$$
\begin{aligned}
\mathcal{Q}=V_1\oplus V_{-1}, \ \ \phi V_1= V_{-1},\ \ {\rm dim}V_1={\rm dim}V_{-1}=m-2.
\end{aligned}
$$
In particular, according to the principal curvatures of
Example \ref{ex:3.9}, we know that Example \ref{ex:3.9} is contained in this case.

If $\lambda=\frac{\alpha\lambda-2}{2\lambda-\alpha}$, then $\lambda=\frac{\alpha\pm\sqrt{\alpha^2-4}}{2}$.
According to Lemma \ref{lemma:5.5} and the assumption that $M$ has
four distinct principal curvatures, we must have
$$
\begin{aligned}
\mathcal{Q}=V_{\frac{\alpha-\sqrt{\alpha^2-4}}{2}}\oplus V_{\frac{\alpha+\sqrt{\alpha^2-4}}{2}}, \ \ \phi V_{\frac{\alpha-\sqrt{\alpha^2-4}}{2}}=V_{\frac{\alpha-\sqrt{\alpha^2-4}}{2}},\ \ \phi V_{\frac{\alpha+\sqrt{\alpha^2-4}}{2}}=V_{\frac{\alpha+\sqrt{\alpha^2-4}}{2}}.
\end{aligned}
$$
It follows that $M$ satisfies $S\phi=\phi S$, which means that $M$ has isometric Reeb flow.
According to Theorem \ref{thm:2.2}, we know that $M$ is an open part of Example \ref{ex:3.5}.

If $\alpha=2$, then by the assumption that $M$ has four distinct principal curvatures,
there is another one constant principal curvature $\lambda\in\sigma(\mathcal{Q})$ and
$\lambda\notin\{2,0,1\}$. For any $X\in V_\lambda$, by Lemma \ref{lemma:5.5},
we have $S\phi X=\phi X$, so $1\in\sigma(\mathcal{Q})$.
It means that
$$
\begin{aligned}
\mathcal{Q}=V_\lambda\oplus V_1, \ \ \phi V_\lambda\subset V_1.
\end{aligned}
$$

If $\lambda=-1$, then we take unit vector fields $X\in V_{-1}$ and $Y=JX\in V_1$ into Cartan's formula \eqref{eqn:call2g}, we have
$$
\begin{aligned}
0&=(-2)\Big[1+2g(\phi X,Y)^2-2g(AX,Y)^2-2g(AX,JY)^2\\
&\qquad\qquad +g(AX,X)g(AY,Y)+g(AX,JX)g(AY,JY)\Big]\\
&=-6\Big(1-g(AX,X)^2-g(AX,JX)^2\Big).
\end{aligned}
$$
It deduces that $g(AX,X)^2+g(AX,JX)^2=1$, which means that $AX\in\{X,JX\}$ for any $X\in V_{-1}$.
Thus $A(V_{-1}\oplus\phi V_{-1})=(V_{-1}\oplus\phi V_{-1})$
and $A(V_1\ominus \phi V_{-1})=(V_1\ominus \phi V_{-1})$.
Assume that ${\rm dim}V_1>{\rm dim}V_{-1}$, from the fact that
$V_1\ominus \phi V_{-1}$ is $J$-invariant, one can choose unit vector fields
$Y_1,JY_1\in V_1\ominus \phi V_{-1}$ such that $AY_1=Y_1$ and $AJY_1=-JY_1$.
Now, we take unit vector fields $X=X_1\in V_{-1}$ and $Y=Y_1\in (V_1\ominus \phi V_{-1})$
into Cartan's formula \eqref{eqn:call2g}, with the use of $g(AX_1,Y_1)=g(AX_1,JY_1)=0$ and
$AY_1=Y_1$, it holds
\begin{equation}\label{eqn:5.5sq}
0=-2(1+g(AX_1,X_1)).
\end{equation}
On the other hand, we take $X=X_1\in V_{-1}$ and $Y=JY_1\in (V_1\ominus \phi V_{-1})$ into Cartan's formula \eqref{eqn:call2g}, with the use of $g(AX_1,Y_1)=g(AX_1,JY_1)=0$ and
$AJY_1=-JY_1$, it holds
\begin{equation}\label{eqn:5.6sq}
0=-2(1-g(AX_1,X_1)).
\end{equation}
By \eqref{eqn:5.5sq} and \eqref{eqn:5.6sq}, we get a contradiction. So it holds
$\phi V_{-1}=V_1$ and ${\rm dim}V_1={\rm dim}V_{-1}=m-2$.
In particular, according to the principal curvatures of
Example \ref{ex:3.10}, we know that Example \ref{ex:3.10} is contained in this case.

If $\lambda\neq-1$, by (3) of Lemma \ref{lemma:5.1}, it holds
$AV_\lambda\bot {\rm Span}\{V_\lambda, JV_\lambda\}$.
Now, we take unit vector fields $X\in V_\lambda$ and $Y=JX\in V_1$ into Cartan's formula \eqref{eqn:call2g}, with the use of $g(AX,X)=g(AX,JX)=0$, it holds
$0=3(\lambda-1)$. It contradicts with $\lambda\neq1$.

In {\bf Case-ii}, it holds $\alpha=0$. If $0\in\sigma(\mathcal{Q})$, then there exists unit tangent vector field $X\in\mathcal{Q}$ such that $SX=0$. Now, by Lemma \ref{lemma:5.5}, it follows that $\phi X=0$, which is a contradiction. So, $0\notin\sigma(\mathcal{Q})$.
There are another three distinct principal curvatures $\lambda_1,\lambda_2,\lambda_3$,
such that $\{\lambda_1,\lambda_2,\lambda_3\}=\sigma(\mathcal{Q})$ and
$\lambda_1,\lambda_2,\lambda_3\neq0$. For any $X\in V_{\lambda_i}$, $1\leq i\leq3$,
by Lemma \ref{lemma:5.5}, we have $S\phi X=\frac{-1}{\lambda_i}\phi X$.
It follows that
$\{\lambda_1,\lambda_2,\lambda_3,\frac{-1}{\lambda_1},\frac{-1}{\lambda_2},
\frac{-1}{\lambda_3}\}=\sigma(\mathcal{Q})$,
this implies that $\mathcal{Q}$ have an even number of nonzero distinct principal curvatures.
But, this contradicts with
the assumption that $M$ has four distinct principal curvatures.

We have completed the proof of Theorem \ref{thm:1.1c}.



\vskip 10mm

\begin{flushleft}
Haizhong Li\\
{\sc Department of Mathematical Sciences, Tsinghua University,\\
Beijing, 100084, P.R. China}\\
E-mail: lihz@tsinghua.edu.cn

\vskip 1mm

Hiroshi Tamaru\\
{\sc Department of Mathematics, Graduate School of Science, Osaka Metropolitan University,
3-3-138, Sugimoto, Sumiyoshi-ku, Osaka, 558-8585, Japan}\\
E-mail: tamaru@omu.ac.jp

\vskip 1mm

Zeke Yao\\
{\sc School of Mathematical Sciences, South China Normal University,\\
Guangzhou 510631, P.R. China}\\
E-mail: yaozkleon@163.com

\end{flushleft}

\end{document}